\documentclass[reqno]{amsart}
\usepackage[backend=bibtex, sorting=none, bibstyle=alphabetic, citestyle=alphabetic, sorting=nyt, maxbibnames=99, giveninits=true, isbn=false, url=false, doi=false, eprint=false]{biblatex}  
\renewbibmacro{in:}{}
\bibliography{references.bib}
\usepackage{amssymb, calrsfs, graphics, graphicx, enumerate, enumitem, url, xcolor, hyperref, mathrsfs, listings, comment}
\usepackage{tabularx}
\usepackage[ruled, lined, linesnumbered, longend]{algorithm2e}
\usepackage[justification=centering]{caption}
\newtheorem{theorem}{Theorem}[section]
\newtheorem{lemma}[theorem]{Lemma}

\newtheorem{corollary}[theorem]{Corollary}
\numberwithin{equation}{section}

\theoremstyle{remark}
\newtheorem*{remark}{Remark}

\lstdefinestyle{CStyle}{
    basicstyle=\footnotesize,
    breakatwhitespace=false,         
    breaklines=true,                 
    captionpos=b,                    
    keepspaces=true,                 
    numbers=left,                    
    numbersep=5pt,                  
    showspaces=false,                
    showstringspaces=false,
    showtabs=false,                  
    tabsize=2,
    language=C
}

\title[An explicit log-free zero density estimate ]{An explicit log-free zero density estimate for the Riemann zeta-function}
\author[C. Bellotti]{Chiara  Bellotti}
\address{School of Science\\
The University of New South Wales, Canberra, Australia}
\email{c.bellotti@unsw.edu.au}
\keywords{Riemann zeta function, zero-density estimates, explicit estimates}
\subjclass[2020]{Primary 11M06, 11M26. Secondary 11Y35}
\begin{document}
\maketitle
\begin{abstract}
We will provide an explicit log-free zero-density estimate for $\zeta(s)$ of the form $N(\sigma,T)\le AT^{B(1-\sigma)}$. In particular, this estimate becomes the sharpest known explicit zero-density estimate uniformly for $\sigma\in[\alpha_0,1]$, with $0.985\le \alpha_0\le 0.9927$ and $3\cdot 10^{12}<T\le \exp(6.7\cdot 10^{12})$.
\end{abstract}
\section{Introduction}
Consider the Riemann zeta function $\zeta(s)$ and let $\rho=\beta+i\gamma$ be a zero of $\zeta(s)$ with $0<\operatorname{Re}(\rho)<1$. Studying the zeros of $\zeta(s)$ is a classical problem in number theory and several approaches have been followed throughout the years to have a better understanding of the behaviour of the zeros of $\zeta(s)$ inside the critical strip. Among them, one is finding upper bounds for the number of zeros of $\zeta(s)$ inside a fixed region. These estimates are commonly known as zero-density estimates. More precisely, for fixed $1/2<\sigma<1$ and $T>0$, we define the quantity$$
N(\sigma, T)=\#\{\rho=\beta+i \gamma: \zeta(\rho)=0,0<\gamma<T,\  \sigma<\beta<1\},
$$
which counts the number of non-trivial zeros of $\zeta(s)$ with real part greater than the fixed value $\sigma$. The goal is to find upper bounds for $N(\sigma,T)$ that are as sharp as possible. Trivially, if the Riemann Hypothesis is true then all the non-trivial zeros of $\zeta(s)$ would lie on the half-line $\sigma=1/2$ and hence $N(\sigma,T)=0$ for every $1/2<\sigma<1$.

Zero-density estimates for $\zeta(s)$ often take the form $
N(\sigma, T) \ll T^{B(1-\sigma)}(\log T)^D
$, where $B, D$ are constants which might depend on $\sigma$. There is an extensive literature on such results, some of which will be recalled in the following paragraph. Usually, a large amount of work is done in trying to minimize the exponent $B$ for a wide range of $\sigma$, while allowing $D$ to be a large absolute constant. In particular, it has been conjectured that the estimate $N(\sigma,T)\ll T^{2(1-\sigma)+\epsilon}$ holds uniformly for every $1/2<\sigma<1$ and $T>0$ (this conjecture, also known as the Density Hypothesis, is actually weaker than what one would get if the Riemann Hypothesis is true). However, our goal is to focus on estimates where $D=0$, that is, estimates of the form
\begin{equation}\label{logfreebound}
N(\sigma,T)\ll T^{B(1-\sigma)}.  
\end{equation}
This type of estimate is commonly known as a log-free zero-density estimate. There is a vast number of research works on the log-free zero-density estimates, starting with Linnik \cite{Lin44, Lin44a} and following with many other authors, including Turán \cite{Turan61}, Fogels \cite{Fogels1965}, Bombieri \cite{bombiericrible}, Jutila \cite{jutila77}, Gallagher \cite{gallagher_large_1970}, Graham \cite{graham_linniks_1981} and Heath-Brown \cite{hb92}.\\
For completeness, we mention some of the known results for $N(\sigma,T)$. In 2021, Platt and Trudgian \cite{platt_riemann_2021} verified that the Riemann Hypothesis is true up to height $3\cdot 10^{12}$ and, as a consequence, $N(\sigma,T)=0$ for every $1/2<\sigma<1$, $T\le 3\cdot 10^{12}$. Hence, we can restrict the range of values of $T$ for which we want to find an upper bound for $N(\sigma,T)$ to $T>3\cdot 10^{12}$.

Furthermore, in 2000, Bourgain \cite{Bourgain2000OnLV} proved that the Density Hypothesis is true for $\sigma\ge 25/32$ and, actually, for $25/32\le \sigma<1$ it is possible to get upper bounds for the quantity $N(\sigma,T)$ which are sharper than the Density Hypothesis (see \cite{trudgian2023optimal} and \cite{pintz2023remark} for more recent results).    
However, although the Density Hypothesis is known to be true when $\sigma$ is sufficiently close to $1$, the sharpest uniform non-explicit bound for $N(\sigma,T)$ which holds for every $1/2<\sigma\le 1$ and $T>0$ is $N(\sigma,T)\ll T^{12(1-\sigma)/5+\varepsilon}$, which can be obtained (\cite{Ivic2003TheRZ}) by combining Ingham's estimate \cite{Ingham1940ONTE} $N(\sigma,T)\ll T^{3(1-\sigma)/(2-\sigma)}\log^ 5T $ for $\frac{1}{2}\le\sigma\le \frac{3}{4}$ with Huxley's one \cite{Huxley1971/72} $N(\sigma,T)\ll T^{3(1-\sigma)/(3\sigma-1)}\log^{44}T $, which holds for $\frac{3}{4}\le\sigma\le 1$. Another important result due to Ingham is a connection between bounds for $\zeta(s)$ on the half line and $N(\sigma,T)$. More precisely, he proved \cite{ingham1937difference} that if $\zeta\left(\frac{1}{2}+i t\right)\ll t^{c+\epsilon}$ for a certain constant $c$, then $
N(\sigma, T)\ll T^{(2+4 c)(1-\sigma)}(\log T)^5
$. In particular, the Lindelöf Hypothesis $\zeta\left(\frac{1}{2}+i t\right)\ll t^\epsilon$ implies  the Density Hypothesis.\\

While non-explicit estimates for $N(\sigma,T)$ are widely studied, much less is known about explicit estimates for $N(\sigma,T)$.  Ingham's estimate has been made explicit by Ramaré \cite{ramare2016explicit} and then improved by Kadiri, Lumley and Ng \cite{KADIRI201822}. When $\sigma$ is really close to the $1/2$ line, an explicit version of Selberg's estimate \cite{selberg1946contribution} is provided by the work of Simonič in \cite{SIMONIC2020124303}. When instead $\sigma$ is really close to $1$ and $T$ quite large, the author provided an explicit version of the upper bound of the form $ N(\sigma,T)\ll T^{B(1-\sigma)^{3/2}}(\log T)^D$ \cite{bellotti2023explicit}. Explicit log-free zero density estimates for Dirichlet $L$-Functions have been found by Thorner and Zaman in \cite{thorner2023explicit}. The estimate found in \cite{thorner2023explicit} holds also for $\zeta(s)$, but in this case the exponent is too big to be used in some applications which might benefit from the existence of a log-free zero density estimate for $\zeta(s)$ and which we will list in the following paragraph.\medskip\\

The aim of this paper is to provide a new explicit log-free zero density estimate for $\zeta(s)$ of the form \eqref{logfreebound}. Furthermore, this estimate becomes the sharpest known explicit zero-density estimate for $\zeta(s)$ uniformly for $\alpha_0\le \sigma\le 1$ and $T\in(T_0,T_1]$, whose values are listed in Table \ref{valuesCB} in Appendix \ref{numerical}. The whole range of $T$ will be $3\cdot 10^{12}<T\le \exp(6.7\cdot 10^{12})$, while $\alpha_0$ is a fixed value in $[0.985,0.9927]$, and depends on the interval $(T_0,T_1]$ in which we are working.
\begin{theorem}\label{th1}
    For every $T\in(T_0,T_1]$ and $\sigma\in[\alpha_0,1]$ as displayed in Table \ref{valuesCB}, the following estimate holds:
    \[
    N(\sigma,T)\le C T^{B(1-\sigma)},
    \]
    where $C=C(\alpha_0,T_0,T_1)$ and $B=B(\alpha_0,T_0,T_1)$ and numerical values are listed in Table \ref{valuesCB} in Appendix \ref{numerical}.
\end{theorem}
In particular, if we want a uniform upper bound which holds for every $3\cdot 10^{12}<T\le \exp(6.7\cdot 10^{12})$ and $\sigma\in [0.9927,1]$, the following estimate holds:
\begin{equation*}
    N(\sigma,T)\le C T^{B(1-\sigma)},
\end{equation*}
with $B=1.448$ and $C=1.62\cdot 10^{11}$. The constant $C$ is quite big; however, as we explained before, Theorem \ref{th1} provides much sharper estimates depending on the size of $T$. An analysis of the values of both $C$ and $B$ listed in Table \ref{valuesCB} shows that when $T$ increases, $C$ increases, but on the contrary $B$ decreases quite consistently. This balance between increasing values of $C$ and decreasing values of $B$ when $T$ becomes larger allow us to get sharp upper bounds for $N(\sigma,T)$ in a certain region on the left of the zero-free region for $\zeta(s)$. The python program used to compute $C$ and $B$ is available on github at \cite{Bellotti_Log-free}. Furthermore, a detailed analysis of the computational aspects of this work can be found in Appendix \ref{numerical}.
\begin{remark}
    The method can be adapted to get log-free zero-density estimates for $\zeta(s)$ for values of $\sigma$ which are smaller than $0.985$, and which can be sharper than the current known estimates of the type \eqref{kadirilumleyng}, provided that $T$ is in sufficiently large ($T>\exp(90)$). However, we limit ourselves to provide a log-free zero-density estimate for $\sigma\ge 0.985$, as values of $\sigma $ which are really close to $1$ are usually needed in the majority of the applications.
\end{remark}

We will divide the proof of Theorem \ref{th1} in two cases. Given $\rho=\beta+i\gamma$ a zero of $\zeta(s)$, if $0<\gamma\le 0.447\log T$ then $\gamma\le 3\cdot 10^{12}$, since we are working with $T\le \exp(6.7\cdot 10^{12})$. Hence $N(\sigma,t)=0$ for $t\le 0.447\log T$ and $T\le e^{6.7\cdot 10^{12}}$, since the Riemann Hypothesis has been verified up to $3\cdot 10^{12}$ by Platt and Trudgian \cite{platt_riemann_2021}. From now on, we then consider $\gamma>0.447\log T$.  
\begin{remark}
The above choice of splitting the proof of Theorem \ref{th1} in two cases and using the partial verification of the Riemann Hypothesis when $\gamma\le 0.447\log T$ prevents us from finding a log-free zero-density estimate for $\zeta(s)$ which holds for every $T$. However, this is not a big limitation, since when $T$ is sufficiently large, estimates of the form $N(\sigma,T)\ll T^{B(1-\sigma)^{3/2}}\log^{D}T$ become sharper.
\end{remark}
\subsection{Importance and applications}
Log-free zero-density estimates of the form \eqref{logfreebound} are particularly powerful when $\sigma$ is sufficiently close to $1$ and close to the zero-free region, as $N(\sigma,T)$ is absolutely bounded by a constant when $\sigma\rightarrow 1$. More precisely, consider both the usual estimate with the log-factor of the form
\begin{equation}\label{logestbound}
    N(\sigma,T)\ll T^{B_1(1-\sigma)}\log^DT
\end{equation}
and the log-free zero density estimate 
\begin{equation}\label{logfreeestbound}
    N(\sigma,T)\ll T^{B_2(1-\sigma)},
\end{equation}
where $B_1,B_2,D$ are absolute constants which may depend on $\sigma$. 
Given a certain positive constant $\lambda$ which is large enough so that $1-\lambda/\log T$ is outside the zero-free region, we want to bound the quantity $N\left(1-\frac{\lambda}{\log T},T\right)$. Using \eqref{logestbound} we have
\begin{equation}\label{comingfromlog}
    N\left(1-\frac{\lambda}{\log T},T\right)\ll e^{\lambda B_1}\log^DT,
\end{equation}
while by \eqref{logfreeestbound} it follows that
\begin{equation}\label{comingfromlogfree}
    N\left(1-\frac{\lambda}{\log T},T\right)\ll e^{\lambda B_2}.
\end{equation}
If we compare \eqref{comingfromlog} and \eqref{comingfromlogfree}, it is evident that \eqref{comingfromlogfree} is much sharper than the first estimate, as the quantity $N\left(1-\frac{\lambda}{\log T},T\right)$ is absolutely bounded by a constant (which depends on $B_2$), while in \eqref{comingfromlog} the upper bound contains a factor $\log^D T$ which increases with $T$ and hence it is not absolutely bounded.\\

Theorem \ref{th1}, combined with the results in \cite{bellotti2023explicit} can lead to improvements in several other topics in computational number theory, such as the error term in the Prime Number Theorem (see \cite{yangdj, fiori_sharper_2023}), primes between consecutive powers (see \cite{CULLYHUGILL2023100,michaela2024error}), Ramanujan's inequality (see \cite{platt_error_2020, michaelavonmangoldt}).
\subsection{Comparison with existent estimates}
In this subsection, we try to compare Theorem \ref{th1} with the existing explicit zero-density estimates. Since we are working with $T\le \exp(6.7\cdot 10^{12})$, we just need to compare Theorem \ref{th1} with Kadiri-Lumley-Ng's zero-density estimate \cite{KADIRI201822} of the form
\begin{equation}\label{kadirilumleyng}
 N(\sigma,T)\le C_1T^{\frac{8}{3}(1-\sigma)}\log^{5-2\sigma}T+C_2\log^2T,   
\end{equation}
where $C_1, C_2$ are absolute constants which depend on $\sigma$. The values of $C_1,C_2$ have been slightly improved by Johnston and Yang in \cite{yangdj} and we will refer to these new values throughout the comparison with the estimate found in Theorem \ref{th1}.\\
To verify that the estimate found in Theorem \ref{th1} is always sharper than \eqref{kadirilumleyng} in all the intervals $\sigma\in[\alpha_0,1]$ and $T\in(T_0,T_1]$ listed in Table \ref{valuesCB} in Appendix \ref{numerical}, we just need to verify that in each of the above intervals the following relation holds:
\begin{equation}\label{suffcondition}
   CT^{B(1-\sigma)}\le C_1T^{\frac{8}{3}(1-\sigma)}\log^{5-2\sigma}T\qquad \forall \sigma\in[\alpha_0,1],\  T\in(T_0,T_1], 
\end{equation}
where the values of $C,B$ are displayed in Table \ref{valuesCB} in Appendix \ref{numerical} and $C_1$ is the same as in \eqref{kadirilumleyng}. If \eqref{suffcondition} is true, in particular Theorem \ref{th1} is sharper than \eqref{kadirilumleyng} for each of the intervals $\sigma\in[\alpha_0,1]$, $T\in(T_0,T_1]$, since in \eqref{kadirilumleyng} there is an additional positive term $C_2\log^2 T$, which does not appear in \eqref{suffcondition}.\\
To verify \eqref{suffcondition}, we use two different techniques depending on whether $T_1\le \exp(500)$.\\
For $T_1\le \exp(500)$, being the exponent $B>\frac{8}{3}$ (the values of $B$ can be found in Table \ref{valuesCB} in Appendix \ref{numerical}), we rewrite \eqref{suffcondition} as 
\begin{equation*}
   T^{(B-\frac{8}{3})(1-\sigma)}\le \frac{C_1}{C}\log^{5-2\sigma}T, 
\end{equation*}
and then we apply the logarithm, getting
\[
\left(B-\frac{8}{3}\right)(1-\sigma)\le \log\left(\frac{C_1}{C}\right)\cdot \frac{1}{\log T}+(5-2\sigma)\cdot\frac{\log\log T}{\log T}.
\]
The last inequality is equivalent to 
\begin{equation}\label{sigmacond}
    \sigma\ge \frac{1}{\left(B-\frac{8}{3}\right)-2\frac{\log\log T}{\log T}}\cdot \left(B-\frac{8}{3}-\log\left(\frac{C_1}{C}\right)\cdot \frac{1}{\log T}-5\frac{\log\log T}{\log T}\right).
\end{equation}
At this point, it remains to verify that, for every $T\in(T_0,T_1]$, with $(T_0,T_1]$ fixed, the RHS of \eqref{sigmacond} is always less than the $\alpha_0$ corresponding to that particular interval $(T_0,T_1]$. Simple numerical calculations show that \eqref{sigmacond} holds for every $T\in(T_0,T_1]$ and $\sigma\in[\alpha_0,1]$.\\
When $T_0>\exp(500)$, we observe that the exponent $B$ in \eqref{suffcondition} given by Theorem \ref{th1} in each of the intervals $(T_0,T_1]$ is smaller than $\frac{8}{3}$, and hence the factor $T^{B(1-\sigma)}$ is always smaller than $T^{\frac{8}{3}(1-\sigma)}$. Furthermore, for a fixed $\sigma$, the RHS in \eqref{suffcondition} is an increasing function in $T$ and the exponent of the log-factor is $5-2\sigma\ge 3$. Hence, we just need to verify that $C_1\log^3(T_0)\ge C$ in each interval $(T_0,T_1]$, where $C_1$ is the same as in \eqref{kadirilumleyng} and $C$ is given by Theorem \ref{th1}. Again, a straightforward calculation shows that this is always true in the intervals we are considering.\\
Another remark is that, when $T$ is sufficiently large, that is, $T\ge \exp(1000)$, the exponent $B$ is less than $2$ which is the exponent conjectured in the Density Hypothesis. This makes sense, since it has already been proved non-explicitly that when $\sigma$ is sufficiently close to $1$, we can get upper bounds for $N(\sigma,T)$ which are sharper than the one conjectured in the Density Hypothesis, as we already mentioned before.\\
To give an idea of the size of the improvements of Theorem \ref{th1} on \eqref{kadirilumleyng}, we just list a few cases, keeping $\sigma$ fixed and letting $T$ vary or vice versa.
\begin{table}[h]
\def\arraystretch{1.3}
\centering
\caption{Comparison between Theorem \ref{th1} and estimate \eqref{kadirilumleyng} }
\begin{tabular}{|c|c|c|}
\hline
$\sigma$ & $T$  & Improvement \\
\hline
$0.9930$ &$10^{13}$  & $13.4\%$\\
\hline 
$0.9930$ &$\exp(46.2)$  & $58.4\%$\\
\hline 
$0.9930$ &$\exp(170.2)$  & $97.4\%$\\
\hline 
$0.9900$ &$\exp(46.2)$  & $4.1\%$\\
\hline 
$0.9900$ &$\exp(170.2)$  & $96.2\%$\\
\hline 
$0.9850$ &$\exp(90)$  & $16.2\%$\\
\hline 
$0.9900$ &$\exp(90)$  & $72.6\%$\\
\hline 
$0.9930$ &$\exp(90)$  & $86.0\%$\\
\hline 
\end{tabular}
\label{ex1}
\end{table}
\section{Background and useful lemmas}\label{background} We list some results that will be used throughout the proof of Theorem \ref{th1}. As we already mentioned in the previous section, since the Riemann Hypothesis has been verified up to height $3\cdot 10^{12}$ by Platt and Trudgian \cite{platt_riemann_2021}, we can restrict to $T> 3\cdot 10^{12}$.
\subsection{Zeros and bounds of $\zeta(s)$}
We start recalling some important results about the number of zeros inside a circle centered at a point of the critical strip close to the $1$-line.
\begin{lemma}\label{lemmangen}
  Let $\sigma \geq 1,0<r \leq 1, T\ge 3\cdot 10^{12} , t \in \mathbb{R},3\cdot 10^{12}\le|t| \leq T$. Define
$$
n(r, \sigma+i t):=\#\{\rho=\beta+i \gamma: 0<\beta<1, \zeta(\rho)=0,|\sigma+i t-\rho| \leq r\} .
$$
The bound $$n(r, \sigma+i t) \leq n(r, 1+i t) \leq n(2 r, 1+r+i t) \leq r(2 \log ( T)-1)+4$$ holds.
\end{lemma}
\begin{proof}
    See  Lemma 2.9 of \cite{thorner2023explicit} with $q=1$.
\end{proof}
\begin{lemma}\label{pr211q1}
Let $3\cdot 10^{12}\le |t| \leq T$, and $0<r<\frac{1}{2}$.  The following bound holds
$$
\begin{aligned}
\sum_{\substack{\zeta(\rho)=0 \\
|1+i t-\rho| \leq r}} \operatorname{Re}\left(\frac{1}{1+r+i t-\rho}\right) & \leq \frac{\log ( T)+4.7098}{4+8 r}+\frac{\left(\frac{4}{\pi}-1\right) \log \left(1+r^{-1}\right)}{1+2 r}+2.6908 \\
& +\frac{8 r}{(1+2 r)^2}(r(2 \log ( T)-1)+4)+r^{-1}.
\end{aligned}
$$
\end{lemma}
\begin{proof}
    Use Proposition 2.11 of \cite{thorner2023explicit} with $q=1$.
    \end{proof}
    An immediate corollary is the following result for the number of zeros close to a point of the $1$-line.
    \begin{corollary}\label{boundnear1torner}Using the same notations as in Lemma \ref{lemmangen}, the following estimate holds for $3\cdot 10^{12}\le |t| \leq T$:
        $$
\begin{aligned}
n(r, 1+i t) \leq 2 r\left(\frac{\log ( T)+4.7908}{4+8 r}+\frac{\left(\frac{4}{\pi}-1\right) \log \left(1+r^{-1}\right)}{1+2 r}+2.6908\right. \\
\left.\quad+\frac{8 r}{(1+2 r)^2}(r(2 \log (T)-1)+4)+r^{-1}\right) .
\end{aligned}
$$In particular, if $$\frac{1}{a\log T}\le r\le \frac{1}{b}$$ we get
\[
n(r, 1+i t)\le \frac{\log T}{b}\left(0.5+\frac{32}{b^2}\right)+\frac{0.573803}{b}\log b+\frac{2.4+5.3816}{b}+\frac{64}{b^2}+2-\frac{272}{b^3}.
\]
    \end{corollary}
    \begin{proof}
    The first part follows from Lemma \ref{pr211q1}.
        If $$\frac{1}{a\log T}\le r\le \frac{1}{b}$$we get
        \begin{equation*}
            \begin{aligned}
&n(r, 1+i t) \leq 2 r\left(\frac{\log ( T)+4.7908}{4+8 r}+\frac{\left(\frac{4}{\pi}-1\right) \log \left(1+r^{-1}\right)}{1+2 r}+2.6908\right. \\&
\left.\quad+\frac{8 r}{(1+2 r)^2}(r(2 \log (T)-1)+4)+r^{-1}\right) \\&\le r(0.5\log T+2.4)+0.573803 r\log\left(\frac{1}{r}\right)+
5.3816r+64 r^2 + r^3 (32 \log T - 272)+2\\&\le r\left( \log T\left(0.5+\frac{32}{b^2}\right)+0.573803\log b+2.4+5.3816+\frac{64}{b}-\frac{272}{b^2}\right)+2\\&\le \frac{\log T}{b}\left(0.5+\frac{32}{b^2}\right)+\frac{0.573803}{b}\log b+\frac{2.4+5.3816}{b}+\frac{64}{b^2}+2-\frac{272}{b^3}, \end{aligned}
        \end{equation*}
        since $32\log T-272\ge 0$ for $T\ge 3\cdot 10^{12}$.
       
    \end{proof}
    Corollary \ref{boundnear1torner} can be used to determine the following result.
\begin{lemma}\label{numberzerosinintervalwithno2}
    Given $0.985\le \alpha_0\le\alpha\le 1$, the number of non trivial zeros for the Riemann zeta function $\zeta(s)$ inside the square \[
\alpha\le \sigma\le 1,\qquad |t-T|\le 1-\alpha,
\]
with $T\ge T_0$ fixed, is bounded by the following quantity depending on $\alpha$:
 \[\tilde{n}(\alpha)=b_1(1-\alpha)\log T+b_2,\]
 where the values of $b_1=b_1(\alpha_0,T_0)$ and $b_2=b_2(\alpha_0,T_0)$ can be found in Table \ref{alpha0b1b2} in Appendix \ref{numerical}.
\end{lemma}
\begin{proof}
First, we observe that the square
\[
\alpha\le \sigma\le 1,\qquad |t-T|\le 1-\alpha
\]
is contained inside the circle of center $1+iT$ and radius 
\[
r=\sqrt{2}(1-\alpha).
\]
Hence, the number of zeros inside the square is less than the number of zeros inside the circle in which the square is contained. The conclusion follows by an application of Corollary \ref{boundnear1torner} with $r=\sqrt{2}(1-\alpha)$ to $n(r,1+iT)$.

\end{proof}
Other important tools for the proof of Theorem \ref{th1} are the bounds for the Riemann zeta function on the half line. 
\begin{lemma}\label{onehalf}
The following estimates hold:
    \begin{equation*}
\left|\zeta\left(\frac{1}{2}+i t\right)\right| \leq\left\{\begin{array}{ll}
1.461 & \text { if } 0 \leq|t| \leq 3 \\ \\
0.618|t|^{1 / 6} \log |t| & \text { if }3<|t| \le \exp(105)\\ \\
 66.7 t^{27 / 164}  & \text { if } |t| > \exp(105).
\end{array}\right.
\end{equation*}
\end{lemma}
In order of appearance, these estimates are due to Hiary \cite{hiary_explicit_2016}, Hiary--Patel--Yang \cite{hiary_improved_2024}, Patel--Yang \cite{patel2023explicit}.\\
 Finally, we list the currently best-known explicit zero-free regions for the Riemann zeta function. Although all of them hold for every $|T|\ge 3$, we indicate the range of values for $|T|$ for which each of them is the sharpest one.
\begin{itemize}
    \item For $3\cdot 10^{12}< |T|\le e^{46.2}$ the classical zero-free region \cite{mossinghoff2024explicit} is the widest one:
    \begin{equation}\label{classicalzerofree}
        \sigma\ge 1-\frac{1}{5.558691 \log |T|}=:1-\nu_1(|T|).
    \end{equation}
    \item For $ e^{46.2}<|T|\le e^{170.2}$ the currently largest known zero-free region is \begin{equation}\label{zerofreeintermediate}
\sigma>1-\frac{0.04962-0.0196 /(J(|T|)+1.15)}{J(|T|)+0.685+0.155 \log \log |T|}=:1-\nu_2(|T|),
\end{equation}
where $J(T):=\frac{1}{6} \log T+\log \log T+\log 0.618$. As per \cite{YANG2024128124}, this result is obtained by substituting Theorem 1.1 of \cite{hiary_improved_2024} in Theorem 3 of \cite{ford_zero_2022} and observing that $J(T)<\frac{1}{4} \log T+1.8521$ for $T\geq 3$.
\item For $e^{170.2}< |T|\le e^{481958}$, the sharpest zero-free region \cite{YANG2024128124} is the Littlewood one:
\begin{equation}\label{littlewoodzerofree}
    \sigma\ge 1-\frac{\log\log |T|}{21.233\log |T|}=:1-\nu_3(|T|).
\end{equation}
\item Finally, for $|T|> e^{481958}$, the Korobov--Vinogradov zero-free region \cite{BELLOTTI2024128249} becomes the largest one: \begin{equation}\label{kvzerofree}
     \sigma\ge  1-\frac{1}{53.989\log^{2/3}|T|(\log\log |T|)^{1/3}}=:1-\nu_4(|T|).
\end{equation}
\end{itemize}
\subsection{Arithmetic functions}Now, we recall some explicit results involving some arithmetic functions. In particular, estimates for the divisor function $d(n)$ and for the Möbius function $\mu(n)$ will play an important role in our argument. We will use the following notation: if $f=\mathcal{O}^*(g)$ then $|f|\le g$.
\begin{lemma}[\cite{CullyHugill2019TwoED}, Theorem 2]\label{michaela-tim}
For $x \geq 2$ we have
$$
\sum_{n \leq x} d^2(n)=D_1 x \log ^3 x+D_2 x \log ^2 x+D_3 x \log x+D_4 x+\mathcal{O}^*\left(9.73 x^{\frac{3}{4}} \log x+0.73 x^{\frac{1}{2}}\right)
$$
where
$$
D_1=\frac{1}{\pi^2}, \quad D_2=0.745 \ldots, \quad D_3=0.824 \ldots, \quad D_4=0.461 \ldots
$$
are exact constants. Furthermore, for $x \geq x_j$ we have
$$
\sum_{n \leq x} d^2(n) \leq K x \log ^3 x
$$
where one may take $\left\{K, x_j\right\}$ to be, among others, $\left\{\frac{1}{4}, 433\right\}$ or $\{1,7\}$.
\end{lemma}
Lemma \ref{michaela-tim} will be used in two different cases, when $x=VW=T^{v+w}$ and when $x=W^2=T^{2w}$, where $v,w$ are parameters we will define in Section \ref{zerodetector} and whose values depend on the range for $T$, $T_0\le T\le T_1$, in which we are working. Throughout the paper, we will use the following notation:
\begin{equation*}
    \sum_{n \leq x} d^2(n) \leq d_1 x \log ^3 x,
\end{equation*}
where $d_1=d_1(v,w,T_0,T_1)$ is equal to $d_{1,1}$ if $x=VW$, while $d_1=d_{1,2}$ if $x=W^2$. The numerical values for $d_{1,1}$ and $d_{1,2}$ can be found in Table \ref{dconstants} in Appendix \ref{numerical} .
By Lemma \ref{michaela-tim}, we can easily deduce the following result.
\begin{lemma}\label{d2overn}
    For every fixed $x$, the following estimate holds:
    \[
    \sum_{n\le x}\frac{d^2(n)}{n}\le d_2\log^4X,
    \]
    where $d_2=d_2(v,w,T_0,T_1)$ is equal to $d_{2,1}$ if $x=VW$ or $d_{2,2}$ if $x=W^2$. The numerical values can be found in Table \ref{dconstants} in Appendix \ref{numerical}. 
\end{lemma}
\begin{proof}
By partial summation and Theorem \ref{michaela-tim} we get
    \begin{equation*}
        \begin{aligned}
            &\sum_{n\le X}\frac{d^2(n)}{n}=\frac{1}{X}\sum_{n\le X}d^2(n)+\int_1^X\left(\sum_{n\le u}d^2(n)\right)\frac{1}{u^{2}}\text{d}u\\&\le d_1\log^3X+\int_1^X\frac{d_1\log^3u}{u}\text{d}u\\&= d_1\log^3X+\frac{d_1}{4}\log^4X\le d_2\log^4X.
        \end{aligned}
    \end{equation*}
\end{proof}
\begin{lemma}[\cite{Ramare15} Corollary 1.10]\label{rammulog}
     For any real number $x \geq 1$ and any positive integer $q$, we have
$$
0\le \sum_{\substack{n \leq x,(n, q)=1}} \mu(n) \log (x / n) / n \leq 1.00303 \cdot q / \varphi(q).
$$
\end{lemma}
Now, define
\begin{equation*}
c_0=\tilde{c}+\sum_{p \geq  2} \frac{\log p}{p(p-1)}=1.332582275733 \ldots,
\end{equation*}
where $\tilde{c}$ is Euler's constant.
\begin{lemma}[\cite{ramare_explicit_2019}, Theorem 3.1]\label{estwithphi}
For $D>1$, we have
$$
\sum_{d \leq  D} \frac{\mu^2(d)}{\varphi(d)}=\log D+c_0+\mathcal{O}^*\left(\frac{11}{\sqrt{D \log D}}\right),
$$
and equally for $D>1$, we have
$$
\sum_{d \leq  D} \frac{\mu^2(d)}{\varphi(d)}=\log D+c_0+\mathcal{O}^*\left(\frac{61 / 25}{\sqrt{D}}\right).
$$
\end{lemma}
Finally, define
\begin{equation*}
G_q(D)=\sum_{\substack{d \leq D,(d, q)=1}} \frac{\mu^2(d)}{\varphi(d)}.
\end{equation*}
\begin{lemma}[\cite{ramare_explicit_2017}, Theorem 1.2.]\label{ramaregd} When $D \geq 1$, we have $$G_1(D)=\log D+c_0+\mathcal{O}^*(3.95 / \sqrt{D}).$$ When $D>1$, we have $$G_1(D)=\log D+c_0+\mathcal{O}^*(18.4 / \sqrt{D \log D}).$$
\end{lemma}
\subsection{Miscellaneous}
We recall an explicit estimate for the $\Gamma $ function (\cite{olver_asymptotics_1974}, p.294).
\begin{lemma}[Explicit Stirling formula]\label{stirling}
For $z=\sigma+it$, with $|\arg z| <\pi$ we have
\[
\left|\Gamma(z)\right|\le (2\pi)^{1/2}|t|^{\sigma-\frac{1}{2}}\exp\left(-\frac{\pi}{2}|t|+\frac{1}{6|z|}\right).
\]
\end{lemma}
Finally, a fundamental role in the proof of Theorem \ref{th1} will be played by Halász--Montgomery inequality.
\begin{theorem}[\cite{Ivic2003TheRZ} A.39, A.40]\label{montineq}
  Let $\xi,\varphi_1,\dots,\varphi_R$ be arbitrary vectors in an inner-product vector space over $\mathbb{C}$, where $(a,b)$ will be the notation for the inner product and $||a||^2=(a,a)$. Then
  \[
  \sum_{r\le R}|(\xi,\varphi_r)|\le ||\xi||\left(\sum_{r,s\le R}|(\varphi_r,\varphi_s)|\right)^{1/2}
  \]
  and
   \[
  \sum_{r\le R}|(\xi,\varphi_r)|\le ||\xi||^2\max_{r\le R}\sum_{s\le R}|(\varphi_r,\varphi_s)|.
  \]
\end{theorem}

\section{Construction of the zero-detector}\label{zerodetector}
Throughout the proof, we will often refer to some techniques used by Pintz in a recent work \cite{pintz_new_2019} on some new non-explicit zero-density estimates for some $L$-functions.\\
We start defining the following Barban-Vehov weights $\left(\psi_d\right)_d$ and $\left(\theta_d\right)_d$ by
$$
\psi_d=\left\{\begin{array}{ll}
\mu(d) & 1 \leq d \leq U, \\ \\
\mu(d) \frac{\log (V / d)}{\log (V / U)} & U<d \leq V, \\ \\
0 & d>V,
\end{array} \right.
$$
and
$$\theta_d=\left\{\begin{array}{ll}
\mu(d) \frac{\log (W / d)}{\log W} & 1 \leq d \leq W, \\ \\
0 & d>W.
\end{array}\right.$$
Then, we denote
\begin{equation}\label{definepsisquare}
    \Psi(n)=\sum_{d \mid n} \psi_d,\qquad \Theta(n)=\sum_{d \mid n} \theta_d.
\end{equation}
We observe that $\theta_d$ is a special case of $\psi_d$ with $U=1$ and $V=W$. Also, $\Psi(n)=0$ for $2\le n\le U$.\\We aim to find some explicit bounds for sums involving both $\Psi$ and $\Theta$. In order to do that, we introduce two other quantities due to Graham \cite{GRAHAM197883}. 
Assume that $1 \leq  z_1 \leq  z_2$, and we let
$$
\begin{aligned}
\Lambda_i(d) & =\mu(d) \log \left(z_i / d\right) & & \text { if } d \leq  z_i \\
& =0 & & \text { if } d>z_i .
\end{aligned}
$$
We observe that, if $z_1<z_2$, and $z_1=U$, $z_2=V$, we have exactly
\begin{equation}\label{psiandlam}
    \psi_d=\frac{\Lambda_2(d)-\Lambda_1(d)}{\log(V/U)}.
\end{equation}
Furthermore, if we put $z_1=W$, we have
\begin{equation*}
    \theta_d=\frac{\Lambda_1(d)}{\log W}.
\end{equation*}
\begin{lemma}\label{grahamz1z2small}
Using the above notations, the following estimate holds for $z_1z_2<N$ and $1\le z_1\le z_2\le N$:
\begin{equation*}
\sum_{1 \leq  n \leq  N}\left(\sum_{d \mid n} \Lambda_1(d)\right)\left(\sum_{e \mid n} \Lambda_2(e)\right)\le 1.0061(N-1) \left(\log z_1+1.333+\frac{11}{\sqrt{z_1\log z_1}}\right)+N.
\end{equation*}.
\end{lemma}
\begin{proof}
Following Graham's argument \cite{GRAHAM197883} for the case $z_1z_2<N$, we have

    \begin{equation}\label{firststep}
\begin{aligned}
&\sum_{1<n \leq  N}\left(\sum_{d \mid n} \Lambda_1(d)\right)\left(\sum_{e \mid n} \Lambda_2(e)\right)\\
&\le (N-1) \sum_{d\le z_1, e\le z_2} \frac{\Lambda_1(d) \Lambda_2(e)}{[d, e]}+\sum_{d\le z_1, e\le z_2}\left|\Lambda_1(d)\right|\left|\Lambda_2(e)\right|.
\end{aligned}
\end{equation}
Trivially,
\begin{equation}\label{errterm}
\begin{aligned}
    &\sum_{d\le z_1, e\le z_2}\left|\Lambda_1(d)\right|\left|\Lambda_2(e)\right|\le \sum_{d \leq  z_1} \log \left(\frac{z_1}{d}\right) \sum_{e \leq  z_2} \log \left(\frac{z_2}{e}\right)\\& \le( z_1-\log(z_1)-1)( z_2-\log(z_2)-1)\le z_1z_2\le N.
\end{aligned}
\end{equation}
Now, we analyze the main term. By Lemma \ref{rammulog} and Lemma \ref{estwithphi}, we have
\begin{equation}\label{mainterm}
\begin{aligned}
&\sum_{d\le z_1, e\le z_2} \frac{\Lambda_1(d) \Lambda_2(e)}{[d, e]} \\&
\quad=\sum_{d\le z_1, e\le z_2} \frac{\Lambda_1(d) \Lambda_2(e)}{d e} \sum_{r \mid(d, e)} \varphi(r) \\&
\quad=\sum_{r \leq  z_1} \frac{\mu^2(r) \varphi(r)}{r^2}\left(\sum_{\substack{m \leq  z_1 / r \\
(m, r)=1}} \frac{\mu(m)}{m} \log \left(\frac{z_1}{m r}\right)\right)\left(\sum_{\substack{n \leq  z_2 / r \\
(n, r)=1}} \frac{\mu(n)}{n} \log \left(\frac{z_2}{n r}\right)\right)\\&\le \sum_{r \leq  z_1} \frac{\mu^2(r) \varphi(r)}{r^2} \cdot (1.00303)^2\cdot \frac{r^2}{\varphi^2(r)}\\&\le 1.0061 \sum_{r \leq  z_1} \frac{\mu^2(r)}{ \varphi(r)}\\&\le 1.0061\cdot \left(\log z_1+1.333+\frac{11}{\sqrt{z_1\log z_1}}\right).
\end{aligned}
\end{equation}
Combining \eqref{firststep}, \eqref{errterm}, \eqref{mainterm}, we get
\begin{equation*}
    \sum_{1<n \leq  N}\left(\sum_{d \mid n} \Lambda_1(d)\right)\left(\sum_{e \mid n} \Lambda_2(e)\right)\le 1.0061(N-1)  \left(\log z_1+1.333+\frac{11}{\sqrt{z_1\log z_1}}\right)+N.
\end{equation*}
\end{proof}
\begin{lemma}\label{psi2caseeasy}
   Let $\Psi$ the function defined in \eqref{definepsisquare}. The following estimate holds when $N>UV$:
    \begin{equation*}
        \sum_{n=1}^N\Psi^2(n)\le \frac{4}{\log^2 (V/U)} 1.0061\left((N-1)\left(\log U+1.333+\frac{11}{\sqrt{U\log U}}\right)+N\right).
    \end{equation*}
\end{lemma}
\begin{proof}
        By \eqref{psiandlam}, 
        \begin{equation*}
            \begin{aligned}
                &\log^2(V/U)\Psi^2(n)=\log^2(V/U)\left(\sum_{d \mid n} \psi_d\right)\left(\sum_{e \mid n} \psi_e\right)\\&=\left(\sum_{d|n}(\Lambda_2(d)-\Lambda_1(d))\right)\left(\sum_{e|n}(\Lambda_2(e)-\Lambda_1(e))\right)\\&=\left[\sum_{d|n}\sum_{e|n}\Lambda_2(d)\Lambda_2(e)-\sum_{d|n}\sum_{e|n}\Lambda_2(d)\Lambda_1(e)-\sum_{d|n}\sum_{e|n}\Lambda_1(d)\Lambda_2(e)+\sum_{d|n}\sum_{e|n}\Lambda_1(d)\Lambda_1(e)\right]\\&\le 4\max_{i,j=1,2}\left|\sum_{d|n}\sum_{e|n}\Lambda_i(d)\Lambda_j(e)\right|.
            \end{aligned}
        \end{equation*}
        Hence, by Lemma \ref{grahamz1z2small}, we can conclude that
        \begin{equation*}
            \begin{aligned}
                &\sum_{n=1}^N\Psi^2(n)\le \frac{4}{\log^2 (V/U)}\left((N-1)\cdot 1.0061\cdot \left(\log U+1.333+\frac{11}{\sqrt{U\log U}}\right)+N\right).
            \end{aligned}
        \end{equation*}
    \end{proof}
In particular, for our purpose we will choose $U=T^{u}$ and $V=T^{v}$, with $T_0\le T\le T_1$ (numerical values of $u,v$ depending on the range of $T$ can be found in Table \ref{uxvw} in Appendix \ref{numerical}). Under these assumptions, whenever $N>UV$, Lemma \ref{psi2caseeasy} implies
    \begin{equation}\label{specialcase}
        \begin{aligned}
            &\sum_{n=1}^{N}\Psi^2(n)\le d_3\cdot \frac{N}{\log T},
        \end{aligned}
    \end{equation}
    where $d_3=d_3(u,v,T_0,T_1)$ and numerical values for $d_3$ can be found in Table \ref{dconstants} in Appendix \ref{numerical}.\\
    Finally, we recall a result due to Ramaré and Zuniga-Altmerman about a sum involving $\Psi(n)$.
    \begin{lemma}[\cite{ramaré2024l2bound}, Corollary 1.2]\label{ramarepreprint}
   The following estimate holds:
   
   \begin{equation}\label{ramaresebnew}
\sum_{n=1}^{N} \frac{\Psi^2(n)}{n}\le 3.09 \frac{\log N}{\log \left(z_2 / z_1\right)} \frac{1.084(t+1)+1.301\left(1+t^2\right)-0.116}{t-1} .
\end{equation}
where $z_2=z_1^t$ and for any $N\ge z_1\ge 100$. In particular, at $t=2$, and $z_1=z$, we have 
\[
\sum_{n=1}^{N} \frac{\Psi^2(n)}{n}\le 30\frac{\log N}{\log z}.
\]
\end{lemma}
\subsection{Setting up the argument}
We start defining some parameters that will be used throughout the proof. Given a fixed $T\ge 3\cdot 10^{12}$, we define the following quantities:
\begin{equation}\label{choiceparam}
    \mathscr{L}:=\log T,\quad U=T^u,\quad  V=T^v,\quad W=T^w, \quad X=T^x,
\end{equation} 
where $0<u<v$ and $w, x>0$, with $w<u$, $u+v<x$, are suitably chosen positive constants whose values can be found in Table \ref{uxvw} in Appendix \ref{numerical} and depend on both the range of $T\in (T_0,T_1]$ and $\alpha_0\le\sigma$. We split the range of values for $T$ into several sub-intervals $(T_0,T_1]$, in order to have sharper estimates in each of these smaller ranges. The criteria with which the sizes of the different intervals are chosen will be explained in Appendix \ref{numerical}.\\
From now on, we will present the argument in terms of $x,u,v,w$. Then, depending on the range of $T$, we just need to substitute the values in the table inside the argument to get the final numerical estimate of Theorem \ref{th1} for each of the intervals. 
\begin{lemma}\label{lowerbounddetector}
 If $\rho$ is a non-trivial zero of $\zeta(s)$ with $\operatorname{Re}(\rho)> \alpha_0$ and $\gamma>0.447\log T$, $T\in(T_0,T_1]$, $T_0\ge 3\cdot 10^{12}$, then
$$
\left|\sum_{n=1}^{\infty}\Psi(n)\Theta(n) n^{-\rho}\left(e^{-n / X}-e^{-n \mathscr{L}^2 / U}\right)\right| \ge c_1,
$$
where $c_1=c_1(\alpha_0,x,u,v,w,T_0,T_1)$ and numerical values can be found in Table \ref{cconstants} in Appendix \ref{numerical}.
\end{lemma}
\begin{proof}
    As per \cite{pintz_new_2019}, we have
    \begin{equation*}
        S(X)=\sum_{n=1}^{\infty}\Psi(n)\Theta(n) n^{-\rho}e^{-n / X}=\int_{1-i \infty}^{1+i \infty} \zeta(s+\rho) M(s+\rho) \Gamma(s) X^s \text{d} s,
    \end{equation*}
    where 
    \[
M(s)=\sum_{v \leq V} \sum_{w \leq W} \psi_v \theta_w [v, w]^{-s}.
\]
    Moving the line of integration to the the line $\operatorname{Re}(s)=\frac{1}{2}-\beta$, we get
    \begin{equation}\label{S(x)}
    \begin{aligned}
        S(X)&=\zeta(\rho)M(\rho)+M(1)\Gamma(1-\rho)X^{1-\rho}+\frac{1}{2\pi i}\int_{\frac{1}{2}-\beta-i \infty}^{\frac{1}{2}-\beta+i \infty} \zeta(s+\rho) M(s+\rho) \Gamma(s) X^s \text{d} s\\&=M(1)\Gamma(1-\rho)X^{1-\rho}\\&\quad +\frac{1}{2\pi i}\int_{- \infty}^{+ \infty} \zeta\left(\frac{1}{2}+i\gamma+it\right) M\left(\frac{1}{2}+i\gamma+it\right) \Gamma\left(\frac{1}{2}-\beta+it\right) X^{\frac{1}{2}-\beta+it} \text{d} t.
    \end{aligned}
    \end{equation}
First, since by Lemma \ref{d2overn} one has
\begin{equation*}
    \begin{aligned}
        &\left|M\left(1\right)\right|=\left|\sum_{v \leq V} \sum_{w \leq W} \psi_v \theta_w [v, w]^{-1}\right|\le \sum_{v \leq V} \sum_{w \leq W} [v, w]^{-1}\\&\le \sum_{n\le VW}d^{2}(n)\cdot n^{-1} \le   d_2\log^4(VW),
    \end{aligned}
\end{equation*}
using $X=T^x$, $VW=T^{v+w}$, we get
\begin{equation}\label{firstfromresidue}
    \begin{aligned}
       &| M(1)\Gamma(1-\rho)X^{1-\rho}|\le \frac{ d_2\cdot (2\pi)^{1/2}e^{\frac{1}{6\gamma}}\log^4(VW)X^{1-\beta}}{e^{\frac{\pi \gamma}{2}}\gamma^{\beta-1+\frac{1}{2}}}\\& \le \frac{ d_2\cdot (u+w)^4(2\pi)^{1/2}e^{\frac{1}{6\gamma}}(\log T)^4T^{x(1-\beta)}}{e^{\pi\cdot 0.447 \log T/2}(0.447\log T)^{\beta-1+\frac{1}{2}}}.
    \end{aligned}
\end{equation}
Now, we estimate the second term of \eqref{S(x)}.\\
By Lemma \ref{d2overn},
\begin{equation*}
    \begin{aligned}
        &\left|M\left(\frac{1}{2}+i(\gamma+t)\right)\right|=\left|\sum_{v \leq V} \sum_{w \leq W} \psi_v \theta_w [v, w]^{-(\frac{1}{2}+i(\gamma+t))}\right|\le \sum_{v \leq V} \sum_{w \leq W} [v, w]^{-1/2}\\&\le \sum_{n\le VW}d^{2}(n)\cdot n^{-1/2} \le   d_2\cdot (VW)^{1/2}\log^4(VW).
    \end{aligned}
\end{equation*}
Hence, being $VW=T^{v+w}$, $\log^4(VW)=(v+w)^4\log^4T$, $X=T^x$,  the second term of \eqref{S(x)} becomes
\begin{equation}\label{termsx}
    \begin{aligned}
        & \frac{1}{2\pi}\left|\int_{- \infty}^{+ \infty} \zeta\left(\frac{1}{2}+i\gamma+it\right) M\left(\frac{1}{2}+i\gamma+it\right) \Gamma\left(\frac{1}{2}-\beta+it\right) X^{\frac{1}{2}-\beta+it} \text{d} t\right|\\&\le \frac{ d_2}{2\pi}\cdot (v+w)^4T^{(v+w)/2}\log^4T\cdot T^{x(\frac{1}{2}-\beta)}\int_{- \infty}^{+ \infty} \left|\zeta\left(\frac{1}{2}+i\gamma+it\right)\right|\left| \Gamma\left(\frac{1}{2}-\beta+it\right)\right|\text{d} t.
    \end{aligned}
\end{equation}
Now, it remains to estimate the integral
\[
\int_{- \infty}^{+ \infty} \left|\zeta\left(\frac{1}{2}+i\gamma+it\right)\right|\left| \Gamma\left(\frac{1}{2}-\beta+it\right)\right|\text{d} t.
\]
We split the integral in three intervals of integration, that is,
\[
\int_{- \infty}^{+ \infty} \left|\zeta\left(\frac{1}{2}+i\gamma+it\right)\right|\left| \Gamma\left(\frac{1}{2}-\beta+it\right)\right|\text{d} t=\int_{-\infty}^{-3-\gamma}+\int_{-3-\gamma}^{3-\gamma}+\int_{3-\gamma}^{\infty},
\]
and we estimate each term individually.\\
By Lemma \ref{onehalf} and Lemma \ref{stirling},
\begin{equation}\label{firstint}
    \begin{aligned}
     &\int_{- \infty}^{-3-\gamma} \left|\zeta\left(\frac{1}{2}+i\gamma+it\right)\right|\left| \Gamma\left(\frac{1}{2}-\beta+it\right)\right|\text{d} t\\&\le (2\pi)^{1/2}\cdot e^{1/6|1/2-\beta|} \int_{-\infty}^{-3-\gamma}\left|\zeta\left(\frac{1}{2}+i\gamma+it\right)\right||t|^{-\beta}e^{-\frac{\pi}{2}|t|}dt \\&\le (2\pi)^{1/2}\cdot e^{1/6|1/2-\alpha_0|}\cdot 66.7 \int_{-\infty}^{-3-\gamma}|t+\gamma|^{ 27/164}|t|^{-\beta}e^{-\frac{\pi}{2}|t|}dt\\&\le (2\pi)^{1/2}\cdot e^{1/6|1/2-\alpha_0|}\cdot 66.7 \int_{-\infty}^{-3-\gamma}\left|2t\right|^{ 27/164}|t|^{-\beta}e^{-\frac{\pi}{2}|t|}dt\\&\le (2\pi)^{1/2}\cdot e^{1/6|1/2-\alpha_0|}\cdot 66.7 \int_{3+\gamma}^{+\infty}\left|2t\right|^{ 27/164}|t|^{-\beta}e^{-\frac{\pi}{2}|t|}dt\\&\le (2\pi)^{1/2}\cdot e^{1/6|1/2-\alpha_0|}\cdot 66.7 \int_{3+0.447\log(T_0)}^{+\infty}(2t)^{ 27/164}t^{-\alpha_0}e^{-\frac{\pi}{2}t}dt\le 10^{-10},
    \end{aligned}
\end{equation}
where we use $\gamma>0.447\log T$, $3\cdot 10^{12}\le T_0$, $\beta\ge \alpha_0\ge 0.985$.\\Similarly,
\begin{equation}\label{secondint}
    \begin{aligned}
     &\int_{-3-\gamma}^{3-\gamma} \left|\zeta\left(\frac{1}{2}+i\gamma+it\right)\right|\left| \Gamma\left(\frac{1}{2}-\beta+it\right)\right|\text{d} t\\&\le (2\pi)^{1/2}\cdot e^{1/6|1/2-\beta|} \int_{-3-\gamma}^{3-\gamma}\left|\zeta\left(\frac{1}{2}+i\gamma+it\right)\right||t|^{-\beta}e^{-\frac{\pi}{2}|t|}dt \\&\le (2\pi)^{1/2}\cdot e^{1/6|1/2-\alpha_0|}\cdot1.461\int_{-3-\gamma}^{3-\gamma}|t|^{-\beta}e^{-\frac{\pi}{2}|t|}dt\\&\le (2\pi)^{1/2}\cdot e^{1/6|1/2-\alpha_0|}\cdot1.461\int_{\gamma-3}^{3+\gamma}t^{-\alpha_0}e^{-\frac{\pi}{2}t}dt\le 10^{-8}.
    \end{aligned}
\end{equation}
Finally, we consider
\[
\int_{3 - \gamma}^{\infty} \left|\zeta\left(\frac{1}{2}+i\gamma+it\right)\right|\left| \Gamma\left(\frac{1}{2}-\beta+it\right)\right|\text{d} t = \int_{3 - \gamma}^0 + \int_0^\infty = I_1 + I_2.
\]
We start with $I_2$. By Lemma \ref{onehalf} and Lemma \ref{stirling}, 
\begin{equation}\label{inti2}
    \begin{aligned}
        &I_2=\int_{0}^{\infty} \left|\zeta\left(\frac{1}{2}+i\gamma+it\right)\right|\left| \Gamma\left(\frac{1}{2}-\beta+it\right)\right|\text{d} t\\&\le 66.7\int_{0}^{1}|t+\gamma|^{ 27/164}\left| \Gamma\left(\frac{1}{2}-\beta+it\right)\right|dt+(2\pi)^{1/2} e^{1/6} 66.7 \int_{1}^{\infty}|t+\gamma|^{ 27/164}|t|^{-\beta}e^{-\frac{\pi}{2}|t|}dt\\&\le 66.7\cdot 1.01243\cdot \gamma^{27/164}\cdot 1.808+ 197.515 \int_{1}^{\infty}(t+\gamma)^{ 27/164}t^{-\alpha_0}e^{-\frac{\pi}{2}t}dt\\&\le 122.092\gamma^{27/164}+197.515 \cdot \gamma^{27/164}\int_{1}^{\infty}\left(\frac{t}{\gamma}+1\right)^{ 27/164}t^{-\alpha_0}e^{-\frac{\pi}{2}t}dt\\&\le 140.297\gamma^{27/164}.
    \end{aligned}
\end{equation}
Similarly,
\begin{equation}\label{inti1}
    \begin{aligned}
        &I_1=\int_{3-\gamma}^{0} \left|\zeta\left(\frac{1}{2}+i\gamma+it\right)\right|\left| \Gamma\left(\frac{1}{2}-\beta+it\right)\right|\text{d} t\\&\le 66.7\int_{-1}^{0}|t+\gamma|^{ 27/164}\left| \Gamma\left(\frac{1}{2}-\beta+it\right)\right|dt+197.515 \int_{3-\gamma}^{-1}|t+\gamma|^{ 27/164}|t|^{-\beta}e^{-\frac{\pi}{2}|t|}dt\\&\le 66.7\cdot 1.01243\cdot \gamma^{27/164}\cdot 1.808+ 197.515\int_{1}^{\gamma-3}(t+\gamma)^{ 27/164}t^{-\alpha_0}e^{-\frac{\pi}{2}t}dt\\&\le 120.092\gamma^{27/164}+197.515 \gamma^{27/164}\int_{1}^{\infty}\left(\frac{t}{\gamma}+1\right)^{ 27/164}t^{-\alpha_0}e^{-\frac{\pi}{2}t}dt\\&\le 140.297\gamma^{27/164}.
    \end{aligned}
\end{equation}
Hence,
\begin{equation}\label{thirdint}
    \int_{3 - \gamma}^{\infty} \left|\zeta\left(\frac{1}{2}+i\gamma+it\right)\right|\left| \Gamma\left(\frac{1}{2}-\beta+it\right)\right|\text{d} t \le 280.594\gamma^{27/164}.
\end{equation}

Combining \eqref{firstint}, \eqref{secondint}, \eqref{thirdint}, being $|\gamma|\le T$, we can conclude that \eqref{termsx} is
\begin{equation}\label{secondtermint}
\begin{aligned}
        &\le  \frac{d_2\cdot 280.594}{2\pi}\cdot (v+w)^4T^{(v+w)/2}\log^4T\cdot T^{x(\frac{1}{2}-\beta)} \cdot  \gamma^{27/164}+10^{-6}\\&\le  44.66\cdot d_2\cdot (v+w)^4T^{(v+w)/2}\log^4T\cdot T^{x(\frac{1}{2}-\beta)} \cdot  T^{27/164}\\&\le  44.66\cdot d_2\cdot (v+w)^4T^{(v+w)/2+x(\frac{1}{2}-\alpha_0)+27/164}\log^4T.
\end{aligned}
\end{equation}
Hence, combining \eqref{firstfromresidue} and \eqref{secondtermint}, we have
\begin{equation}\label{estimatesx}
\begin{aligned}
      &|S(X)|\le \frac{ d_2\cdot (u+w)^4(2\pi)^{1/2}e^{\frac{1}{6\gamma}}(\log T)^4T^{x(1-\beta)}}{e^{\pi\cdot 0.447 \log T/2}(0.447\log T)^{\beta-1+\frac{1}{2}}}\\&+ 44.66\cdot d_2\cdot (v+w)^4T^{(v+w)/2+x(\frac{1}{2}-\alpha_0)+27/164}\log^4T. 
\end{aligned}
\end{equation}
Now, it remains to estimate the quantity
\begin{equation*}
    \sum_{n=1}^{\infty}\Psi(n)\Theta(n) n^{-\rho}e^{-n \mathscr{L}^2 / U}.
\end{equation*}
Since $\Psi(n)=0$ for $2\le n\le U$, one has
\begin{equation}\label{eqsulu}
    \begin{aligned}
        &S(U/\mathscr{L}^2)=\sum_{n=1}^{\infty}\Psi(n)\Theta(n) n^{-\rho}e^{-n \mathscr{L}^2 / U}=e^{- \mathscr{L}^2 / U}+\sum_{n>U}\Psi(n)\Theta(n) n^{-\rho}e^{-n \mathscr{L}^2 / U}.
    \end{aligned}
\end{equation}
Using the trivial inequality,
\[
\left|\sum_{ n> U}\Psi(n)\Theta(n) n^{-\rho}e^{-n \mathscr{L}^2 / U}\right|\le \sum_{n > U}\Psi(n)\Theta(n)n^{-\alpha_0}e^{-n \mathscr{L}^2 / U}.
\]
Furthermore, by definition of $\Psi(n)$, one has, for $n>U$, 
\[
\Psi(n) = \frac{1}{\log (V/U)}\left(\log V\sum_{d | n}\mu(d) - \sum_{d | n}\mu(d)\log d\right) = \frac{\Lambda(n)}{\log(V/U)},
\]
and similarly, $\Theta(n)$ satisfies
\[
\Theta(n) = \sum_{e | n}\mu(e)\frac{\log(W/e)}{\log W} = \frac{\Lambda(n)}{\log W}.
\]
Hence,
\[
\Psi(n)\Theta(n) \le \frac{\log^2 n}{\log W \log (V/U)}.
\]
It follows that
  \begin{equation}\label{tailsul}
    \begin{aligned}
        &\left|\sum_{ n> U}\Psi(n)\Theta(n) n^{-\rho}e^{-n \mathscr{L}^2 / U}\right|\le \sum_{ n> U}\frac{\log^2n}{\log W(\log(V/U))}\cdot n^{-\alpha_0} e^{-n \mathscr{L}^2 / U}\\&\le \frac{3.8}{w(v-u)\log^2T}\int_{n=U}^{\infty}\log^2n\cdot  n^{-0.985}e^{-n \mathscr{L}^2 / U}\text{d}n\\&\le  \frac{3.8}{w(v-u)\log^2T}\int_{n=U}^{\infty}ne^{-n \mathscr{L}^2 / U}\text{d}n\\&\le \frac{3.8}{w(v-u)\log^2T}\frac{e^{-\mathscr{L}^2}\left(\mathscr{L}^2+1\right)}{\frac{\mathscr{L}^4}{U^2}}= \frac{3.8}{w(v-u)\log^2T}U^{2}\cdot \frac{e^{-\mathscr{L}^2}\left(\mathscr{L}^2+1\right)}{\mathscr{L}^4}\\&= \frac{3.8}{w(v-u)\log^2T}e^{-\mathscr{L}(\mathscr{L}-2u)}\cdot \frac{\left(\mathscr{L}^2+1\right)}{\mathscr{L}^4}\le 10^{-20}.
    \end{aligned}
\end{equation}
Combining \eqref{estimatesx}, \eqref{eqsulu} and \eqref{tailsul}, we can conclude that
\begin{equation}\label{constinfront}
    \begin{aligned}
        &\left|\sum_{n=1}^{\infty}\Psi(n)\Theta(n) n^{-\rho}\left(e^{-n / X}-e^{-n \mathscr{L}^2 / U}\right)\right|\\&\ge |e^{- \mathscr{L}^2 / U}|-|S(x)|-\left|\sum_{ n> U}\Psi(n)\Theta(n) n^{-\rho}e^{-n \mathscr{L}^2 / U}\right|&\ge c_1,
    \end{aligned}
\end{equation}
where $c_1=c_1(u,x,v,w,T_0,T_1)$ is a constant whose numerical values are collected in Table \ref{cconstants} in Appendix \ref{numerical}.
\end{proof}
\section{Proof of Theorem \ref{th1}}
Consider the rectangle 
\[
R(\alpha,T)=\{\sigma+it\ |\ \alpha\le\sigma\le 1,\ 0\le t\le T\}
\]
and divide $R(\alpha,T)$ into smaller rectangles:
\begin{equation}\label{rectangles}
\alpha\le \sigma\le 1,\qquad\qquad (n-1)\delta\le t\le n\delta,\qquad \delta=1-\alpha    
\end{equation}
where $n=1,\dots, T/\delta$. Furthermore, we observe that $\delta>\nu_i(T)$ by the properties of the zero-free regions mentioned before in Section \ref{background}, where $i=1,2,3,4,$ depending on the range of $T$ in which we are working. Then, for each of these smaller rectangles, we choose a representative zero arbitrarily.\\
Denote with $J$ the total number of these representative zeros. 
Our main goal is to find an upper bound for the quantity
\begin{equation}\label{principal}
    \left(\sum_{r=1}^J\left|\sum_{n=1}^\infty\Psi(n)\Theta(n) n^{-\rho}\left(e^{-n / X}-e^{-n\mathscr{L}^2/U}\right)\right|\right)^2,
\end{equation}
which, combined with Lemma \ref{lowerbounddetector} will give an upper bound for $J$.\\
The type of estimate for $J$ that we aim to find is of the form 
\begin{equation*}
    J\le c\cdot \frac{X^{2-2\alpha}}{\delta\log T},
\end{equation*}
so that, when combined with Lemma \ref{numberzerosinintervalwithno2} to get an upper bound for the total number of zeros counted by $N(\alpha,T)$, will give us a log-free zero density estimate for $N(\alpha,T)$. More precisely, by Lemma \ref{numberzerosinintervalwithno2}, we will have a final estimate of the following form:
\begin{equation*}
    N(\alpha,T)\le J \tilde{n}(\alpha)\le J (b_1\delta\log T+b_2)\le CX^{2-2\alpha}=C T^{2x(1-\alpha)},
\end{equation*}
where $X=T^x$ for a certain $x$ whose values are listed in Table \ref{uxvw} in Appendix \ref{numerical}.\\
\subsection{Estimate for $J$} In this subsection we will prove the following lemma.
\begin{lemma}\label{upperigma1}
    Under the above conditions the following estimate holds:
    \begin{equation*}
       \left(\sum_{r=1}^J\left|\sum_{n=1}^\infty\Psi(n)\Theta(n) n^{-\rho}\left(e^{-n / X}-e^{-n\mathscr{L}^2/U}\right)\right|\right)^2\le J\cdot (c_2+c_3+c_4)\cdot  \frac{ X^{2-2\alpha}}{\delta\log T}+J^2\cdot  c_5,
    \end{equation*}
    where $c_2,c_3,c_4,c_5$ are constants which depend on $\alpha_0,u,x,v,w,T_0,T_1$ and whose numerical values can be found in Table \ref{cconstants} in Appendix \ref{numerical}.
\end{lemma}
First of  all, we observe that, since $X>U/\mathscr{L}^2$,
\[0<(e^{-n/X}-e^{-n(\mathscr{L}^2/U)})<1.
\]
Applying Lemma \ref{montineq} to \eqref{principal} with
   \[
  \xi_n=\Psi(n)(e^{-n/X}-e^{-n(\mathscr{L}^2/U)})^{1/2}n^{-1/2}  \quad \forall n=1,2,\dots
  \]
and

    \[
\varphi_{r}=\{\varphi_{r,n}\}_{n=1}^{\infty},\quad \varphi_{r,n}=\Theta(n)n^{1/2-\sigma_r-it_r}(e^{-n/X}-e^{-n(\mathscr{L}^2/U)})^{1/2}\qquad \forall n=1,2,\dots,
\]
it follows that \eqref{principal} is 
\begin{equation}\label{afterhalaszmont}
 \begin{aligned}
    &\le \left(\sum_{n=1}^{\infty}\frac{\Psi^2(n)}{n}(e^{-n/X}-e^{-n(\mathscr{L}^2/U)})\right)\left(\sum_{r,s\le J}\left|\sum_{n=1}^\infty\Theta^2(n)n^{1-\rho_r-\overline{\rho_k}}(e^{-n/X}-e^{-n(\mathscr{L}^2/U)})\right|\right).
\end{aligned}
\end{equation}
We will estimate the two factors separately. About the first factor, we prove the following result.
\begin{lemma}\label{an2bn}
    With the same notations as before, the following estimate holds:
    \[
     \sum_{n=1}^{\infty}\frac{\Psi^2(n)}{n}(e^{-n/X}-e^{-n(\mathscr{L}^2/U)})\le d_4,
    \]
    where $d_4=d_4(\alpha_0,u,x,v,w,T_0,T_1)$ and numerical values can be found in Table \ref{dconstants} in Appendix \ref{numerical}.   
\end{lemma}
\begin{proof}
    First, we observe that by our assumptions
\begin{equation*}
    \begin{aligned}
         \sum_{n=1}^{\infty}\frac{\Psi^2(n)}{n}(e^{-n/X}-e^{-n(\mathscr{L}^2/U)})\le \sum_{n=1}^{\infty} \frac{\Psi^2(n)}{n}e^{-n / X}<\sum_{n=1}^{\infty} \frac{\Psi^2(n)}{n}.
    \end{aligned}
\end{equation*}
Since $\Psi(n)=0$ for $2\le n\le U$, we have
    \[
    \sum_{n=1}^{\infty}\frac{\Psi^2(n)}{n}e^{-n/X}\le 1+\sum_{n=U}^{UV}\frac{\Psi^2(n)}{n}e^{-n/X}+\sum_{n=UV}^{\infty}\frac{\Psi^2(n)}{n}e^{-n/X}.
    \]
    We estimate the two quantities separately.\\
    Regarding the first term, by Lemma \ref{ramarepreprint} with $z_1=U=T^u$, $z_2=V=T^v$, $t=v/u$, one has
    \begin{equation*}
        \begin{aligned}
         &\sum_{n=U}^{UV}\frac{\Psi^2(n)}{n}e^{-n/X}\le \sum_{n=U}^{UV}\frac{\Psi^2(n)}{n}\\&\le3.09\cdot \frac{\log(UV)}{\log(V/U)}\cdot \frac{1.301((v/u)^2+1)+1.084((v/u)+1)-0.116}{(v/u)-1}\\&\le 3.09\cdot \frac{u+v}{v-u}\cdot \frac{1.301((v/u)^2+1)+1.084((v/u)+1)-0.116}{(v/u)-1}.
        \end{aligned}
    \end{equation*}
For the second term,
    \[
    \sum_{n=UV}^{\infty}\frac{\Psi^2(n)}{n}e^{-n/X}=\lim_{N\rightarrow \infty}\sum_{n=z_1z_2}^{N}\frac{\Psi^2(n)}{n}e^{-n/X},
    \]
    and, by \eqref{specialcase},
\begin{equation*}
    \begin{aligned}
        &\sum_{n=UV}^{N}\frac{\Psi^2(n)}{n}e^{-n/X}\\&=\frac{1}{Ne^{N/X}}\sum_{n\le N}\Psi^2(n)-\frac{1}{UVe^{UV/X}}\sum_{n\le UV}\Psi^2(n)+\int_{UV}^N\left(\sum_{1\le n\le u}\Psi^2(u)\right)\left(\frac{e^{-u/X}}{Xu}+\frac{e^{-u/X}}{u^2}\right)\text{d}u\\&\le \frac{d_3}{e^{N/X}\log T}+\frac{d_3}{e^{UV/X}\log T}+\frac{d_3}{\log T}\int_{UV}^N\left(\frac{e^{-u/X}}{X}+\frac{e^{-u/X}}{u}\right)\text{d}u.
    \end{aligned}
\end{equation*}
    Now,
    \[
    \int_{UV}^\infty \frac{e^{-u/X}}{X}\le e^{-UV/X}\le 1.
    \]
    Then,
    \begin{equation*}
        \begin{aligned}
            &\int_{UV}^\infty\frac{e^{-u/X}}{u}=\int_{UV}^X\frac{e^{-u/X}}{u}+\int_{X}^\infty\frac{e^{-u/X}}{u}\le \log(X/UV)+1\le 1+(x-(u+v))\log T.
        \end{aligned}
    \end{equation*}
    Hence, for $N\rightarrow\infty$, we have
    \begin{equation*}
         \sum_{n=UV}^{\infty}\frac{\Psi^2(n)}{n}e^{-n/X}\le \frac{d_3}{e^{UV/X}\log T }+\frac{d_3}{\log T}\left(2+(x-(u+v)\log T)\right).
    \end{equation*}
It follows that
\begin{equation*}
    \begin{aligned}
        & \sum_{n=1}^{\infty}\frac{\Psi^2(n)}{n}(e^{-n/X}-e^{-n(\mathscr{L}^2/U)})\\&\le  3.09\cdot \frac{u+v}{v-u}\cdot \frac{1.301((v/u)^2+1)+1.084((v/u)+1)-0.116}{(v/u)-1}\\&\ \ +\frac{d_3}{e^{UV/X}\log T}+\frac{d_3}{\log T}\left(2+(x-(u+v)\log T)\right)\le d_4,
    \end{aligned}
\end{equation*}
where $d_4$ is a constant which depends on $\alpha_0,u,x,v,w,T_0,T_1$ and whose numerical values can be found in Table \ref{dconstants} in Appendix \ref{numerical}.
\end{proof}
Now, we deal with the second factor of \eqref{afterhalaszmont}, that is,
\[
\sum_{r,s\le J}\left|\sum_{n=1}^\infty\Theta^2(n)n^{1-\rho_r-\overline{\rho_k}}(e^{-n/X}-e^{-n\mathcal{L}^2/U})\right|.
\]
We have
\begin{equation*}
    \begin{aligned}
     &\sum_{n=1}^\infty\Theta^2(n)n^{1-\rho_r-\overline{\rho_s}}(e^{-n/X}-e^{-n\mathcal{L}^2/U})\\&=\frac{1}{2\pi i}\int_{2-i\infty}^{2+i\infty}\zeta\left(s+\rho_{r}+\overline{\rho_s}-1\right)G\left(s+\rho_r+\overline{\rho_s}-1\right)\Gamma(s)\left(X^s-\left(\frac{U}{\mathcal{L}^2}\right)^s\right)\text{d}s,   
    \end{aligned}
\end{equation*}
where
\begin{equation}\label{defGs}
G(s)=\sum_{\substack{j \leq W, k \leq W}} \theta_j \theta_k[j, k]^{-s}.
\end{equation}
Moving the line of integration to $\operatorname{Re}(s)=\frac{3}{2}-\beta_r-\beta_s$, we get
\begin{equation}\label{afterresidueth}
    \begin{aligned}
          &\frac{1}{2\pi i}\int_{2-i\infty}^{2+\infty}\zeta\left(s+\rho_{r}+\overline{\rho_s}-1\right)G\left(s+\rho_r+\overline{\rho_s}-1\right)\Gamma(s)\left(X^s-\left(\frac{U}{\mathcal{L}^2}\right)^s\right)\text{d}s\\&=G(1)\Gamma\left(2-\rho_r-\overline{\rho_s}\right)\left(X^{2-\rho_r-\overline{\rho_s}}-\left(\frac{U}{\mathcal{L}^2}\right)^{2-\rho_r-\overline{\rho_s}}\right)\\&+\frac{1}{2\pi i}\int_{-\infty}^{+\infty}\left(\zeta\left(\frac{1}{2}+i(\gamma_r-\gamma_s+t)\right)G\left(\frac{1}{2}+i(\gamma_r-\gamma_s+t)\right)\right.\\&\quad\left.\times \Gamma\left(\frac{3}{2}-\beta_r-\beta_s+it\right)\left(X^{\frac{3}{2}-\beta_r-\beta_s+it}-\left(\frac{U}{\mathcal{L}^2}\right)^{\frac{3}{2}-\beta_r-\beta_s+it}\right)\right)\text{d}t.
    \end{aligned}
\end{equation} 
\subsubsection{Contribution from residues}
We start estimating the overall contribution coming from the residues in \eqref{afterresidueth}, that is,
\begin{equation}\label{contrresidue}
\begin{aligned}
      &\sum_{r,s\le J}\left|G(1)\Gamma(2-\rho_r-\overline{\rho_s})\left(X^{2-\rho_r-\overline{\rho_s}}-\left(\frac{U}{\mathcal{L}^2}\right)^{2-\rho_r-\overline{\rho_s}}\right)\right|\\&\le 2|G(1)|X^{2-2\alpha}\sum_{r,s\le J}\left|\Gamma(2-\rho_r-\overline{\rho_s})\right|.
\end{aligned} 
\end{equation}
The following estimate holds for $|G(1)|$.
\begin{lemma}\label{estg1}
    Given $G(s)$ defined as in \eqref{defGs}, the following estimate holds:
    \[
    |G(1)|\le \frac{d_5}{\log W},
    \]
    where $d_5=d_5(u,x,v,w,T_0,T_1)$ and the numerical values can be found in Table \ref{dconstants} in Appendix \ref{numerical}.
\end{lemma}
\begin{proof}
By definition of $G(s)$, we have
    \[
G(1)=\sum_{\substack{j \leq W, k \leq W}} \theta_j \theta_k[j, k]^{-1}.
\]
Following Graham's argument \cite{GRAHAM197883}, by Lemma \ref{rammulog} and Lemma \ref{ramaregd}, one has
\begin{equation*}
\begin{aligned}
&\sum_{j, k\le W} \frac{\theta_j \theta_k}{[j, k]} \\&
\quad=\sum_{l,k\le W} \frac{\theta_j \theta_k}{jk} \sum_{r \mid(j, k)} \varphi(r) \\&
\quad=\frac{1}{(\log W)^2}\sum_{r \leq  W} \frac{\mu^2(r) \varphi(r)}{r^2}\left(\sum_{\substack{m \leq  W / r \\
(m, r)=1}} \frac{\mu(m)}{m} \log \left(\frac{W}{m r}\right)\right)\left(\sum_{\substack{n \leq  W / r \\
(n, r)=1}} \frac{\mu(n)}{n} \log \left(\frac{W}{n r}\right)\right)\\&\le \frac{1}{(\log W)^2}\sum_{r \leq  W} \frac{\mu^2(r) \varphi(r)}{r^2} \cdot (1.00303)^2\cdot \frac{r^2}{\varphi^2(r)}\\&\le \frac{1}{(\log W)^2}1.0061 \sum_{r \leq  W} \frac{\mu^2(r)}{ \varphi(r)}\\&\le \frac{1}{(\log W)^2}1.0061\left(\log W+ 1.333+\frac{3.95}{\sqrt{W}}\right)\\&\le \frac{d_5}{(\log W)},
\end{aligned}
\end{equation*}
where $d_5=d_5(u,x,v,w,T_0,T_1)$.
\end{proof}
By Lemma \ref{estg1} it follows that \eqref{contrresidue} is
\begin{equation}\label{sumrs}
    \le \frac{ 2d_5}{w\log T}\cdot X^{2-2\alpha}\sum_{r,s\le J}\left|\Gamma\left(2-\rho_r-\overline{\rho_k}\right)\right|.
\end{equation}
    Since
\begin{align*}
    \Gamma(z)=\frac{1}{z} \int_0^{\infty} t^z e^{-t} d t,
\end{align*}
one has
\begin{equation}\label{asymboundgamma}
|\Gamma(z)|\le \frac{1}{|z|} \left|\int_0^{\infty} t^z e^{-t} d t\right|\le \frac{1}{|z|}\int_{0}^\infty t^{x}e^{-t}\text{d}t\le \frac{1}{|z|}\cdot 1.0067, 
\end{equation}
where we used the fact that, as before, $0<x\le 2-2\alpha\le 2-2\alpha_0$, $\alpha_0\ge  0.985$.\\
Then, we consider three  separate cases.
\begin{itemize}
\item When $|\gamma_r-\gamma_s|<\delta$, we have
\begin{equation*}
\begin{aligned}
      &\left|\Gamma\left(2-\rho_r-\overline{\rho_k}\right)X^{2-\rho_r-\overline{\rho_k}}\right|\le X^{2-\beta_j-\beta_k}\cdot \frac{1.0067}{|2-\beta_j-\beta_k|}\\&\le 1.0067\cdot X^{2-2\alpha}\cdot \frac{1}{2\delta} ,
\end{aligned}
\end{equation*}
 being the function $X^{2-y}/(2-y)$ decreasing for $y>1/2$. Since by our assumptions in \eqref{rectangles} for every fixed $\gamma_r$ there is a unique representative zero with imaginary part $\gamma_s$ such that $|\gamma_r-\gamma_s|<\delta$, \eqref{sumrs} restricted to the case $|\gamma_r-\gamma_s|<\delta$  becomes
\begin{equation}\label{singtermledelta}
    \le J\cdot \frac{1.0067\cdot d_5}{w} \frac{X^{2-2\alpha}}{\delta \log T}.
\end{equation}
Multiplying \eqref{singtermledelta} by the bound for the first factor in \eqref{afterhalaszmont} found in Lemma \ref{an2bn}, we have
\begin{equation}\label{boundc2}
     J\cdot c_2 \cdot \frac{X^{2-2\alpha}}{\delta \log T},
\end{equation}
where
\[
c_2=   \frac{1.0067\cdot d_4\cdot d_5}{w}
\]
is a constant depending on $\alpha_0,u,x,v,w,T_0,T_1$ whose numerical values can be found in Table \ref{cconstants} in Appendix \ref{numerical}.
\item Now, we consider the case $\delta<|\gamma_r-\gamma_s|<2$ . By the construction \eqref{rectangles}, for every fixed $\gamma_r$ and for each $n=1\cdots, \delta^{-1}-1$ there is a unique representative zero with imaginary part $\gamma_s$ such that $n\delta\le |\gamma_r-\gamma_s|\le (n+1)\delta$. Hence, by \eqref{asymboundgamma}, \eqref{sumrs} restricted to the case $\delta\le |\gamma_r-\gamma_s|\le 1$ is
\begin{equation}\label{singbet1anddelta}
\begin{aligned}
    &\frac{ 2d_5}{w\log T}\sum_{r\ne s\le J} \left|\Gamma\left(2-\rho_r-\overline{\rho_k}\right)X^{2-\rho_r-\overline{\rho_k}}\right|\\&\le \frac{ 2d_5}{w\log T}\cdot X^{2-2\alpha}\cdot \sum_{r=1}^J\sum_{\delta\le|\gamma_r-\gamma_s|\le 1}\left|\Gamma\left(2-\rho_r-\overline{\rho_k}\right)\right|\\&\le  J\cdot \frac{X^{2-2\alpha}}{\delta\log T}\frac{ 2\cdot d_5\cdot 1.0067}{w} \sum_{n=1}^{\frac{1}{\delta}-1}\frac{1}{n}\\&\le J\cdot \frac{X^{2-2\alpha}}{\delta\log T}\cdot \frac{ 2\cdot d_5\cdot 1.0067}{w}\cdot 1.12\cdot \log((\nu_i(T_1))^{-1}),
\end{aligned}
\end{equation}
where we used the fact that $\delta>\nu_i(T_1)$ ($i=1,2,3$ or $4$ depending on the range $(T_0,T_1]$). Again, multiplying \eqref{singbet1anddelta} by the bound for the first factor in \eqref{afterhalaszmont} given in Lemma \ref{an2bn}, we get
\begin{equation}\label{boundc3}
    J\cdot c_3\cdot \frac{X^{2-2\alpha}}{\delta\log T},
\end{equation}
where 
\[
c_3=\frac{ 2\cdot d_4\cdot d_5\cdot 1.0067\cdot 1.12}{w}\cdot \log(\nu_i(T_1))^{-1}
\]
is always a constant depending on $\alpha_0,u,x,v,w,T_0,T_1$ whose numerical values can be found in Table \ref{cconstants} in Appendix \ref{numerical}.
\item  If $1<|\gamma_r-\gamma_s|\le T$, similarly to what we said for the previous case, by the construction \eqref{rectangles}, for every fixed $\gamma_r$ and for each $n=\delta^{-1},\dots,T\delta^{-1}-1$, there is a unique representative zero with imaginary part $\gamma_s$ such that $n\delta\le |\gamma_r-\gamma_s|\le (n+1)\delta$. Hence, by \eqref{asymboundgamma}, \eqref{sumrs} restricted to the case $1< |\gamma_r-\gamma_s|\le T$ is an application of Lemma \ref{stirling} gives
 \begin{equation}\label{singdeltaT}
        \begin{aligned}
        &\frac{2d_5}{w\log T}\sum_{\substack{r\ne s\le J\\1<|\gamma_r-\gamma_s|\le T}} \left|\Gamma\left(2-\rho_r-\overline{\rho_k}\right)X^{2-\rho_r-\overline{\rho_k}}\right|\\
            &\le \frac{2d_5}{w\log T}\cdot X^{2-2\alpha}\sum_{r=1}^J\sum_{1<|\gamma_r-\gamma_s|\le T}\left|\Gamma\left(2-\rho_r-\overline{\rho_k}\right)\right|\\&\le  \frac{2d_5}{w\log T}\cdot J\cdot X^{2-2\alpha}\cdot (2\pi)^{1/2}\cdot 2.4\cdot \int_{1/\delta}^{T/\delta}e^{1/6(n\delta)}e(n\delta)^{1.5-2\cdot\alpha}\cdot e^{-\pi n\delta/2}\text{d}n\\&\le  J\cdot \frac{X^{2-2\alpha}}{\delta \log T}\frac{2\cdot 2.4\cdot d_5\cdot (2\pi)^{1/2}}{w} \int_{1}^{T}e^{1/6y}y^{1.5-2\cdot\alpha_0}\cdot e^{-\pi y/2}\text{d}y.
        \end{aligned}
    \end{equation}
    As before, multiplying \eqref{singdeltaT} by the quantity $d_4$ found in Lemma \ref{an2bn} we have
    \begin{equation}\label{boundc4}
         J\cdot c_4\cdot \frac{X^{2-2\alpha}}{\delta\log T},
    \end{equation}
where
    \[
    c_4=\frac{2\cdot (2\pi)^{1/2}\cdot 2.4\cdot d_5\cdot d_4}{w}\int_{1}^{T}e^{1/6y}y^{1.5-2\cdot\alpha_0}\cdot e^{-\pi y/2}\text{d}y.
    \]
\end{itemize}
Summing \eqref{boundc2},\eqref{boundc3},\eqref{boundc4} together we can conclude that \eqref{sumrs} becomes
\begin{equation*}
\le  J\cdot (c_2+c_3+c_4)\cdot \frac{X^{2-2\alpha}}{\delta\log T}.
\end{equation*}

\subsubsection{Contribution from integral}
At this point, it remains to study the contribution given by the integral in \eqref{afterresidueth}, that is
\begin{equation*}
    \begin{aligned}
        &\frac{1}{2\pi i}\int_{-\infty}^{+\infty}\left(\zeta\left(\frac{1}{2}+i(\gamma_r-\gamma_s+t)\right)G\left(\frac{1}{2}+i(\gamma_r-\gamma_s+t)\right)\right.\\&\left.\quad\times\Gamma\left(\frac{3}{2}-\beta_r-\beta_s+it\right)\left(X^{\frac{3}{2}-\beta_r-\beta_s+it}-\left(\frac{U}{\mathcal{L}^2}\right)^{\frac{3}{2}-\beta_r-\beta_s+it}\right)\right)\text{d}t.
    \end{aligned}
\end{equation*}
First, we observe that \begin{equation*}
    \left|G\left(\frac{1}{2}+i(\gamma_j-\gamma_k+t)\right)\right| \le \sum_{\substack{j \leq W, k \leq W}} \theta_j \theta_k[j, k]^{-1/2}\le (W^2)^{1/2}\cdot |G(1)|\le  \frac{d_5\cdot W}{\log W},
\end{equation*}
where in the last inequality we used Lemma \ref{estg1}.
Hence,
\begin{equation}\label{remainderrhs}
    \begin{aligned}
       &\left|\frac{1}{2\pi i}\int_{-\infty}^{+\infty}\left(\zeta\left(\frac{1}{2}+i(\gamma_r-\gamma_s+t)\right)G\left(\frac{1}{2}+i(\gamma_r-\gamma_s+t)\right)\right.\right.\\&\left.\left.\quad\times \Gamma\left(\frac{3}{2}-\beta_r-\beta_s+it\right)\left(X^{\frac{3}{2}-\beta_r-\beta_s+it}-\left(\frac{U}{\mathcal{L}^2}\right)^{\frac{3}{2}-\beta_r-\beta_s+it}\right)\right)\text{d}t\right|\\&\le \frac{ d_5}{2\pi}\cdot \frac{W}{\log W}\cdot\left(X^{\frac{3}{2}-\beta_r-\beta_s}+\left(\frac{U}{\mathcal{L}^2}\right)^{\frac{3}{2}-\beta_r-\beta_s}\right)\\&\times \int_{-\infty}^{\infty}\left|\zeta\left(\frac{1}{2}+i(\gamma_r-\gamma_s+t)\right)\right|\left|\Gamma\left(\frac{3}{2}-\beta_r-\beta_s+it\right)\right|\text{d}t.
    \end{aligned}
\end{equation}
It remains to estimate the integral 
\begin{equation}\label{singleintegral}
    \int_{-\infty}^{\infty}\left|\zeta\left(\frac{1}{2}+i(\gamma_r-\gamma_s+t)\right)\right|\left|\Gamma\left(\frac{3}{2}-\beta_j-\beta_k+it\right)\right|\text{d}t.
\end{equation}
We denote $\gamma_r-\gamma_s=:a.$  Furthermore, we just consider the case $0\le a\le T$, since the same contribution is given when $a<0$ after a change of  variable. Hence, from now on, we will work with $a\ge0$. We split \eqref{singleintegral} as
\begin{equation*}
     \int_{-\infty}^{\infty}\left|\zeta\left(\frac{1}{2}+i(\gamma_r-\gamma_s+t)\right)\right|\left|\Gamma\left(\frac{3}{2}-\beta_j-\beta_k+it\right)\right|\text{d}t=\int_{-\infty}^{-3-a}+\int_{-3-a}^{3-a}+\int_{3-a}^{+\infty}
\end{equation*}
and we estimate each term separately.\\
For the first term, similarly to what we did for \eqref{firstint}, an application of Lemma \ref{onehalf} and Lemma \ref{stirling} gives
\begin{equation}\label{firstintrhs}
 \begin{aligned}
     &   J_1=\int_{- \infty}^{-3-a} \left|\zeta\left(\frac{1}{2}+i(\gamma_r-\gamma_s+t)\right)\right|\left|\Gamma\left(\frac{3}{2}-\beta_r-\beta_s+it\right)\right|\text{d} t\\&\le (2\pi)^{1/2}\cdot e^{1/6(2\alpha_0-1.5)} \int_{-\infty}^{-3-a}\left|\zeta\left(\frac{1}{2}+ia+it\right)\right||t|^{1-\beta_r-\beta_s}e^{-\frac{\pi}{2}|t|}dt \\&\le (2\pi)^{1/2}\cdot  e^{1/6(2\alpha_0-1.5)}\cdot 66.7 \int_{-\infty}^{-3-a}|t+a|^{ 27/164}|t|^{1-\beta_r-\beta_s}e^{-\frac{\pi}{2}|t|}dt\\&\le (2\pi)^{1/2}\cdot  e^{1/6(2\alpha_0-1.5)}\cdot 66.7 \int_{-\infty}^{-3-a}\left|2t\right|^{ 27/164}|t|^{1-\beta_r-\beta_s}e^{-\frac{\pi}{2}|t|}dt\\&\le (2\pi)^{1/2}\cdot  e^{1/6(2\alpha_0-1.5)}\cdot 66.7 \int_{3+a}^{+\infty}\left|2t\right|^{ 27/164}|t|^{1-\beta_r-\beta_s}e^{-\frac{\pi}{2}|t|}dt\\&\le (2\pi)^{1/2}\cdot  e^{1/6(2\alpha_0-1.5)}\cdot 66.7 \int_{3}^{+\infty}(2t)^{ 27/164}t^{1-\alpha_0-\alpha_0}e^{-\frac{\pi}{2}t}dt\le 0.24113,
    \end{aligned}
\end{equation}
where we used $\beta\ge\alpha\ge\alpha_0$.\\
Then, an application of Lemma \ref{onehalf} gives
\begin{equation}\label{secondintrhs}
    \begin{aligned}
     &J_2=\int_{-3-a}^{3-a} \left|\zeta\left(\frac{1}{2}+i(\gamma_r-\gamma_s+t)\right)\right|\left|\Gamma\left(\frac{3}{2}-\beta_r-\beta_s+it\right)\right|\text{d} t\\&\le  1.461\int_{-\infty}^{3}\left|\Gamma\left(\frac{3}{2}-\beta_r-\beta_s+it\right)\right|\text{d} t\le 1.461\cdot 4.05206\le 5.921.
    \end{aligned}
\end{equation}
Now, we consider
\[
\int_{3-a}^{+\infty}\left|\zeta\left(\frac{1}{2}+i(\gamma_r-\gamma_s+t)\right)\right|\left|\Gamma\left(\frac{3}{2}-\beta_r-\beta_s+it\right)\right|\text{d} t=\int_{3-a}^a+\int_a^{+\infty}=J_3+J_4.
\]
   For $J_4$, similarly to what we did for $J_1$, we have
    \begin{equation*}
        \begin{aligned}
            &J_4=66.7\int_{ a}^{\infty}\left|t+a\right|^{27/164}\left|\Gamma\left(1-\beta_r-\beta_s+it\right)\right|dt\\&\le (2\pi)^{1/2}\cdot e^{1/6(2\alpha_0-1.5)}\cdot 66.7 \int_{3}^{+\infty}(2t)^{ 27/164}t^{1-2\alpha_0}e^{-\frac{\pi}{2}t}dt\\&\le 0.24113.
        \end{aligned}
    \end{equation*}
    We now shift our attention on $J_3$.\\
    We start considering the case $a>3$. If $3<a\le 4$, then we split $J_3$ as 
    \[
    J_3=\int_{-1}^1+\int_{1}^a,
    \]
    while if $a>4$, then we split $J_3$ into three parts:
    \[
    J_3=\int_{3-a}^{-1}+\int_{-1}^1+\int_{1}^a=J_{3,1}+J_{3,2}+J_{3,3}.
    \]
    Since we are estimating the absolute values of each of the integrals in which we divided $J_3$, we will just estimate $J_3$ when $a>4$ (although for some estimates where we divide by $a$ we use the worse bound $a>3$), as it will give the biggest contribution.\\
     To estimate $J_{3,1}$, an application of Lemma \ref{onehalf} and Lemma \ref{stirling} gives
    \begin{equation*}
        \begin{aligned}
            &J_{3,1}\le 66.7\cdot (2\pi)^{1/2}\cdot e^{1/6(2\alpha_0-1.5)}a^{27/164}\int_{1}^{\infty}\left(t/3+1\right)^{27/164}t^{1-2\alpha_0}e^{-\frac{\pi}{2}t}dt\\&\le 16.329\cdot a^{27/164}.
        \end{aligned}
    \end{equation*}
    The same holds for $J_{3,3}$, which becomes
     \begin{equation*}
        \begin{aligned}
            &J_{3,3}\le 16.329\cdot a^{27/164}.
        \end{aligned}
    \end{equation*}
        About $J_{3,2}$, we use the same argument as for $J_{2,2}$ and we get
    \begin{equation*}
        \begin{aligned}
            &J_{3,2}=66.7\int_{-1}^{1}|t+a|^{27/164}\left|\Gamma\left(\frac{3}{2}-\beta_r-\beta_s+it\right)\right|dt\\&\le 66.7\cdot a^{27/164}\int_{-1}^1|t/3+1|^{27/164}\left|\Gamma\left(\frac{3}{2}-\beta_r-\beta_s+it\right)\right|dt\\&\le 253.419 \cdot a^{27/164}.
        \end{aligned}
    \end{equation*}
    Combining the estimates for $J_{3,1},J_{3,2},J_{3,3}$ when $a>3$, we can conclude that, when $3<a\le T$, 
    \begin{equation}\label{estJ3}
        J_{3}\le 286.077 \cdot a^{27/164}.
    \end{equation}
The same method used to estimate \eqref{singleintegral} when $3<a\le T$ can be applied to the case $0\le a\le 3$. However, the contribution given by \eqref{singleintegral} when $0\le a\le 3$ is smaller. Hence, from now on, we will use the largest contribution which is given when $3<a\le T$.\\
To conclude, combining the estimates we found for $J_1,J_2,J_3,J_4$ with \eqref{remainderrhs} and the fact that $\beta_r,\beta_s\ge\alpha\ge\alpha_0$, it follows that \eqref{remainderrhs} is
   \begin{equation}\label{singleintfactor}
\begin{aligned}
    & \le \frac{ d_5}{2\pi}\cdot \frac{W}{\log W}\cdot\left(X^{\frac{3}{2}-2\alpha_0}+\left(\frac{U}{\mathcal{L}^2}\right)^{\frac{3}{2}-2\alpha_0}\right)\left(J_1+J_2+J_3+J_4\right),
\end{aligned}
\end{equation}
and, multiplying \eqref{singleintfactor} by the first factor in \eqref{afterhalaszmont} one gets
\begin{equation}\label{integralcontriution}
\begin{aligned}
    & \le d_4\cdot \frac{ d_5}{2\pi}\cdot \frac{W}{\log W}\cdot\left(X^{\frac{3}{2}-2\alpha_0}+\left(\frac{U}{\mathcal{L}^2}\right)^{\frac{3}{2}-2\alpha_0}\right)\left(J_1+J_2+J_3+J_4\right)\le c_5,
\end{aligned}
\end{equation}
where $c_5<1$ is a constant depending on $\alpha_0,u,x,v,w,T_0,T_1$ whose values can be found in Table \ref{cconstants} in Appendix \ref{numerical}.\\

\subsection{Conclusion} If we combine Lemma \ref{lowerbounddetector} and Lemma \ref{upperigma1}, it follows that
\begin{equation*}
    \begin{aligned}
        &J^2\cdot c_1^2  \le   J\cdot (c_2+c_3+c_4)\cdot  \frac{ X^{2-2\alpha}}{\delta\log T}+J^2\cdot  c_5,
    \end{aligned}
\end{equation*}
and hence,
\begin{equation*}
    \begin{aligned}
        &J^2\cdot (c_1^2-c_5)  \le   J\cdot (c_2+c_3+c_4)\cdot  \frac{ X^{2-2\alpha}}{\delta\log T},
    \end{aligned}
\end{equation*}
or equivalently,
\begin{equation}\label{upperboundJ}
    \begin{aligned}
        &J  \le  \frac{(c_2+c_3+c_4)}{(c_1^2-c_5)}\cdot  \frac{ X^{2-2\alpha}}{\delta\log T}.
    \end{aligned}
\end{equation}
To conclude the proof of Theorem \ref{th1}, it remains to multiply the upper bound \eqref{upperboundJ} for $J$ by the number of zeros contained in each rectangle \eqref{rectangles}. More precisely, multiplying \eqref{upperboundJ} by the quantity $\tilde{n}(\alpha)$ in Lemma \ref{numberzerosinintervalwithno2} we have
\begin{equation}\label{finallogfreeest}
    \begin{aligned}
        &N(\alpha,T)\le \frac{(c_2+c_3+c_4)}{(c_1^2-c_5)}\cdot  \frac{ X^{2-2\alpha}}{\delta\log T}\cdot (b_1\cdot \delta\cdot \log T+b_2)\\&\le X^{2-2\alpha}\cdot\left(C_1+C_2\right)= CT^{2x(1-\alpha)},
    \end{aligned}
\end{equation}
where
\begin{equation*}
    \begin{aligned}
        C_1&=\frac{(c_2+c_3+c_4)\cdot b_1}{(c_1^2-c_5)},\\
        C_2&=\frac{(c_2+c_3+c_4)\cdot b_2}{(c_1^2-c_5)\delta\log T}\le \frac{(c_2+c_3+c_4)\cdot b_2}{(c_1^2-c_5)\cdot \nu_i(T)\log T}\le \frac{(c_2+c_3+c_4)\cdot b_2}{(c_1^2-c_5)\cdot \nu_i(T_0)\log T_0}.
    \end{aligned}
\end{equation*}
Theorem \ref{th1} follows taking $B=2x$ and $C=C_1+C_2$ in \eqref{finallogfreeest} and replacing $\alpha$ with $\sigma$.
\section*{Acknowledgements}
I would like to thank O. Ramaré and S. Zuniga Alterman for their
insightful comments upon this article and in particular for having shared with me some useful unpublished results about Lemma \ref{ramarepreprint}. I am grateful to N. Ng, A. Zaman, J. Thorner, B. Kerr, and A. Yang for some suggestions upon this paper. Thanks also to my supervisor Timothy S. Trudgian for his support and helpful advice while working on this manuscript.
\appendix
\section{Computation and Numerical Values}\label{numerical}
In this Appendix, we list the numerical values for all the parameters and constants we encountered throughout the proof of Theorem \ref{th1} for every fixed interval $(T_0,T_1]$, and $[\alpha_0,1]$. To compute all the constants defined throughout the proof of Theorem \ref{th1}, one can use the python program available on github at \cite{Bellotti_Log-free} with $\texttt{t}=T_0$, $\texttt{tupp}=T_1$, $\texttt{sigma}=\alpha_0$, and $\texttt{u},\texttt{x},\texttt{v},\texttt{w}$ equal to the corresponding values listed in Table \ref{uxvw} ($\texttt{t},\texttt{tupp},$ $\texttt{sigma},\texttt{u},\texttt{x},\texttt{v},\texttt{w}$ are the variables defined in the python program). Furthermore, the values of $u,x,v,w$ were chosen using a stochastic optimization routine.

We briefly explain our choice for the length of the intervals $(T_0,T_1]$. First of all, none of the intervals $(T_0,T_1]$ intersects the range of values of $T$ for which two different zero-free regions become the sharpest ones (the range of values of $T$ for which each of the known zero-free regions becomes the widest are mentioned in Section \ref{background}).

The most sensitive area in which it is more difficult to get an improvement on \eqref{kadirilumleyng} is the range $3\cdot 10^{12}< T\le\exp(46.2)$ which is just above the Riemann height, and which correspond to the range of $T$ for which the classical zero-free region \eqref{classicalzerofree} is the largest. Hence, when $3\cdot 10^{12}<T\le \exp(46.2)$, we choose $(T_0,T_1]$ so that, if $T_0=\exp(n)$, then $T_1=\exp(n+1)$.

When $\exp(46.2)<T\le \exp(170.2)$, that is, when Ford's zero-free region \eqref{zerofreeintermediate} becomes sharper, a slightly enlargement of the intervals $(T_0,T_1]$ still allows us to get improvements on \eqref{kadirilumleyng}, but with the benefit of dealing with less cases. More precisely, we choose $(T_0,T_1]$ so that, if $T_0=\exp(n)$, then $T_1=\exp(n+10)$.

Then, for $\exp(170.2)<T\le \exp(481958)$, which is the range of $T$ for which the Littlewood zero-free region \eqref{littlewoodzerofree} is the sharpest, we further enlarge the size of the intervals for the same reason as before. More precisely, when $\exp(170.2)<T\le \exp(3000)$, if $T_0=\exp(n)$, then $T_1=\exp(n+500)$. When instead $\exp(3000)<T\le \exp(481958)$, we just consider a single interval. 

When $T>\exp(481958)$, as we mentioned in Section \ref{background}, the Korobov--Vinogradov zero-free region \eqref{kvzerofree} becomes the widest one. Although the range $(\exp(481958),$ $\exp(6.7\cdot 10^{12})]$ is unlikely to be required in many applications, for completeness we provide anyway a log-free zero-density estimate which holds uniformly for every $\sigma\in [0.985,1]$. More precisely, if we take $u=0.3520926$, $x=0.7238738$, $v=0.3616651$, $w=0.0000662$, we get $N(\sigma,T)\le 1.62\cdot 10^{11}\cdot T^{1.448(1-\sigma)}$ for all $\sigma\in [0.985,1]$ and $T\in (\exp(481958),\exp(6.7\cdot 10^{12})]$. When $T\ge \exp(6.7\cdot 10^{12})$, the estimate find by the author in \cite{bellotti2023explicit} is sharper than \eqref{kadirilumleyng} for every $\sigma\in [0.98,1]$. 
\begin{table}[h]
\def\arraystretch{1.3}
\centering
\caption{Values of constants $C,B$ of Theorem \ref{th1}}
\begin{tabular}{|c|c|c|c|c|}
\hline
$T_0$ & $T_1$ & $\alpha_0$ & $C$ & $B$\\
\hline
$3\cdot 10^{12} $ & $\exp(29)$ & $0.9927$ & $37341.72$ & $14.160$ \\
\hline
$\exp(29)$ & $\exp(30)$ & $0.9927$ & $ 37564.18$ & $14.084$ \\
\hline
$\exp(30)$ & $\exp(31)$ & $0.9925$ & $40014.59$ & $13.588$ \\
\hline
$\exp(31)$ & $\exp(32)$ & $0.9923$ & $42840.01$ & $13.108$ \\
\hline
$\exp(32)$ & $\exp(33)$ & $0.9921$ & $46015.02$ & $12.647$ \\
\hline
$\exp(33)$ & $\exp(34)$ & $0.9919$ & $49270.71$ & $12.226$ \\
\hline
$\exp(34)$ & $\exp(35)$ & $0.9918$ & $52249.66$ & $11.856$ \\
\hline
$\exp(35)$ & $\exp(36)$ & $0.9916$ & $55524.80$ & $11.505$ \\
\hline
$\exp(36)$ & $\exp(37)$ & $0.9914$ & $59715.72$ & $11.134$ \\
\hline
$\exp(37)$ & $\exp(38)$ & $0.9912$ & $64228.93$ & $10.782$ \\
\hline
$\exp(38)$ & $\exp(39)$ & $0.9910$ & $68446.65$ & $10.475$\\
\hline
$\exp(39)$ & $\exp(40)$ & $0.9909$ & $71907.32$ & $10.217$\\
\hline
$\exp(40)$ & $\exp(41)$ & $0.9907$ & $76267.80$ & $9.951$\\
\hline
$\exp(41)$ & $\exp(42)$ & $0.9906$ & $80004.48$ & $9.720$\\
\hline
$\exp(42)$ & $\exp(43)$ & $0.9905$ & $83811.09$ & $9.501$\\
\hline
$\exp(43)$ & $\exp(44)$ & $0.9903$ & $ 89209.33$ & $9.259$\\
\hline
$\exp(44)$ & $\exp(45)$ & $0.9901$ & $95233.68$ & $9.020$\\
\hline
$\exp(45)$ & $\exp(46.2)$ & $0.9900$ & $99227.30$ & $8.843$\\
\hline
$\exp(46.2)$ & $\exp(50)$ & $0.9899 $ & $104761.04 $ & $8.634 $\\
\hline
$\exp(50)$ & $\exp(60)$ & $0.9892 $ & $128806.85 $ & $7.936 $\\
\hline
$\exp(60)$ & $\exp(70)$ & $0.9876 $ & $194323.75$ & $6.661 $\\
\hline
$\exp(70)$ & $\exp(80)$ & $0.9860 $ & $278045.07 $ & $5.809$\\
\hline
$\exp(80)$ & $\exp(90)$ & $0.9850 $ & $370655.73$ & $5.216 $\\
\hline
$\exp(90)$ & $\exp(100)$ & $0.9850 $ & $425721.47$ & $4.831 $\\
\hline
$\exp(100)$ & $\exp(110)$ & $0.9850 $ & $488901.14$ & $4.513 $\\
\hline
$\exp(110)$ & $\exp(120)$ & $0.9850 $ & $545744.21$ & $4.264 $\\
\hline
$\exp(120)$ & $\exp(130)$ & $0.9850$ & $629490.27$ & $4.032$\\
\hline
$\exp(130)$ & $\exp(140)$ & $0.9850$ & $694045.43$ & $3.855$\\
\hline
$\exp(140)$ & $\exp(150)$ & $0.9850$ & $771373.78$ & $3.696$\\
\hline
$\exp(150)$ & $\exp(160)$ & $0.9850$ & $825913.07$ & $3.572$\\
\hline
$\exp(160)$ & $\exp(170.2)$ & $0.9850$ & $909966.07$ & $3.448$\\
\hline
$\exp(170.2)$ & $\exp(500)$ & $0.9850$  & $1.12\cdot 10^6$ & $3.337$ \\
\hline
$\exp(500)$ & $\exp(1000)$ & $0.9850$  & $6.23\cdot 10^6$ & $2.152$ \\
\hline
$\exp(1000)$ & $\exp(1500)$ & $0.9850$ & $2.12\cdot 10^7$ & $1.820$ \\
\hline
$\exp(1500)$ & $\exp(2000)$ & $0.9850$ & $6.47\cdot 10^7$ & $1.684$ \\
\hline
$\exp(2000)$ & $\exp(2500)$ & $0.9850$ & $1.73\cdot 10^8$ & $1.610$\\
\hline
$\exp(2500)$ & $\exp(3000)$ & $0.9850$ & $2.56\cdot 10^8$ & $1.577$\\
\hline
$\exp(3000)$ & $\exp(481958)$ & $0.9850$ & $5.76\cdot 10^8$ & $1.551$\\
\hline
\end{tabular}
\label{valuesCB}
\end{table}
\begin{table}[h]
\def\arraystretch{1.3}
\centering
\caption{Values of constants $b_1,b_2$ }
\begin{tabular}{|c|c|c|c|c|}
\hline
$T_0$ & $T_1$ & $\alpha_0$ & $b_1$ & $b_2$   \\
\hline
$3\cdot 10^{12} $ & $\exp(29)$ & $0.9927$ & $0.711$ & $2.113$\\
\hline
$\exp(29)$ & $\exp(30)$ & $0.9927$ & $0.711$ & $2.113$\\
\hline
$\exp(30)$ & $\exp(31)$ & $0.9925$ & $0.711$ & $2.116$\\
\hline
$\exp(31)$ & $\exp(32)$ & $0.9923$ & $0.711$ & $2.119$\\
\hline
$\exp(32)$ & $\exp(33)$ & $0.9921$ & $0.711$ & $2.122$\\
\hline
$\exp(33)$ & $\exp(34)$ & $0.9919$ & $0.712$ & $2.125$\\
\hline
$\exp(34)$ & $\exp(35)$ & $0.9918$ & $0.712$ & $2.127$\\
\hline
$\exp(35)$ & $\exp(36)$ & $0.9916$ & $0.712$ & $2.130$\\
\hline
$\exp(36)$ & $\exp(37)$ & $0.9914$ & $0.712$ & $2.133$\\
\hline
$\exp(37)$ & $\exp(38)$ & $0.9912$ & $0.712$ & $2.136$\\
\hline
$\exp(38)$ & $\exp(39)$ & $0.9910$ & $0.712$ & $2.139$\\
\hline
$\exp(39)$ & $\exp(40)$ & $0.9909$ & $0.713$ & $2.141$\\
\hline
$\exp(40)$ & $\exp(41)$ & $0.9907$ & $0.713$ & $2.144$\\
\hline
$\exp(41)$ & $\exp(42)$ & $0.9906$ & $0.713$ & $2.145$\\
\hline
$\exp(42)$ & $\exp(43)$ & $0.9905$ & $0.713$ & $2.147$\\
\hline
$\exp(43)$ & $\exp(44)$ & $0.9903$ & $0.713$ & $2.150$\\
\hline
$\exp(44)$ & $\exp(45)$ & $0.9901$ & $0.713$ & $2.153$\\
\hline
$\exp(45)$ & $\exp(46.2)$ & $0.9900$ & $0.714$ & $2.155$\\
\hline
$\exp(46.2)$ & $\exp(50)$ & $0.9899 $ & $0.714 $ & $2.156 $\\
\hline
$\exp(50)$ & $\exp(60)$ & $0.9892 $ & $0.715 $ & $2.169 $\\
\hline
$\exp(60)$ & $\exp(70)$ & $0.9876 $ & $0.717 $ & $2.194 $\\
\hline
$\exp(70)$ & $\exp(80)$ & $0.9860 $ & $0.719$ & $2.219$\\
\hline
$\exp(80)$ & $\exp(90)$ & $0.9850 $ & $0.721$ & $2.235$\\
\hline
$\exp(90)$ & $\exp(100)$ &$0.9850 $ & $0.721$ & $2.235$\\
\hline
$\exp(100)$ & $\exp(110)$ & $0.9850 $ & $0.721$ & $2.235$\\
\hline
$\exp(110)$ & $\exp(120)$ & $0.9850 $ & $0.721$ & $2.235$\\
\hline
$\exp(120)$ & $\exp(130)$ & $0.9850 $ & $0.721$ & $2.235$\\
\hline
$\exp(130)$ & $\exp(140)$ & $0.9850$ & $0.721$ & $2.235$ \\
\hline
$\exp(140)$ & $\exp(150)$ & $0.9850$ & $0.721$ & $2.235$ \\
\hline
$\exp(150)$ & $\exp(160)$ & $0.9850$ & $0.721$ & $2.235$ \\
\hline
$\exp(160)$ & $\exp(170.2)$ & $0.9850$ & $0.721$ & $2.235$ \\
\hline
$\exp(170.2)$ & $\exp(500)$ &  $0.9850$ & $0.721$ & $2.235$ \\
\hline
$\exp(500)$ & $\exp(1000)$ & $0.9850$ & $0.721$ & $2.235$   \\
\hline
$\exp(1000)$ & $\exp(1500)$ & $0.9850$ & $0.721$ & $2.235$  \\
\hline
$\exp(1500)$ & $\exp(2000)$ & $0.9850$ & $0.721$ & $2.235$  \\
\hline
$\exp(2000)$ & $\exp(2500)$ & $0.9850$ & $0.721$ & $2.235$\\
\hline
$\exp(2500)$ & $\exp(3000)$ & $0.9850$ & $0.721$ & $2.235$\\
\hline
$\exp(3000)$ & $\exp(481958)$ & $0.9850$ & $0.721$ & $2.235$\\
\hline
\end{tabular}
\label{alpha0b1b2}
\end{table}
\begin{table}[h]
\def\arraystretch{1.3}
\centering
\caption{Choice of parameters in \eqref{choiceparam} }
\begin{tabular}{|c|c|c|c|c|c|}
\hline
$T_0$ & $T_1$ & u & x & v & w \\
\hline
$3\cdot 10^{12} $ & $\exp(29)$ & $2.1781287$ & $7.0798370$ & $4.4744028$ & $0.4808273$ \\
\hline
$\exp(29)$ & $\exp(30)$ & $2.1653098$ & $7.0418787$& $4.4543359$ & $0.4781452$\\
\hline
$\exp(30)$ & $\exp(31)$ & $2.1031623$ & $6.7936945$& $4.2847829$ & $0.4594357$\\
\hline
$\exp(31)$ & $\exp(32)$ & $2.0368399$ & $6.5537594$& $4.1225941$ & $0.4381695$\\
\hline
$\exp(32)$ & $\exp(33)$ & $1.9765605$ & $6.3232916$& $3.9645830$ & $0.4185077$\\
\hline
$\exp(33)$ & $\exp(34)$ & $1.9180226$ & $6.1129209$& $3.8226214$ & $0.3996157$\\
\hline
$\exp(34)$ & $\exp(35)$ & $1.8662754$ & $5.9277629$& $3.7007007$ & $0.3836159$\\
\hline
$\exp(35)$ & $\exp(36)$ & $1.8173525$ & $5.7524453$& $3.5838346$ & $0.3679799$\\
\hline
$\exp(36)$ & $\exp(37)$ & $1.7666469$ & $5.5665117$& $3.4573946$ & $0.3515541$\\
\hline
$\exp(37)$ & $\exp(38)$ & $1.7189733$ & $5.3905275$& $3.3372497$ & $0.3356460$\\
\hline
$\exp(38)$ & $\exp(39)$ & $1.6775968$ & $5.2374475$& $3.2330781$ & $0.3226654$\\
\hline
$\exp(39)$ & $\exp(40)$ & $1.6401586$ & $5.1084977$& $3.1503483$ & $0.3118339$\\
\hline
$\exp(40)$ & $\exp(41)$ & $1.6037571$ & $4.9754347$& $3.0601163$ & $0.3005496$\\
\hline
$\exp(41)$ & $\exp(42)$ & $1.5712431$ & $4.8596030$& $2.9846610$ & $0.2913112$\\
\hline
$\exp(42)$ & $\exp(43)$ & $1.5381371$ & $4.7502479$& $2.9152677$ & $0.2813016$\\
\hline
$\exp(43)$ & $\exp(44)$ & $1.5053035$ & $4.6290398$& $2.8330162$ & $0.2710128$\\
\hline
$\exp(44)$ & $\exp(45)$ & $1.4748658$ & $4.5099741$& $2.7513548$ & $0.2612853$\\
\hline
$\exp(45)$ & $\exp(46.2)$ & $1.4468014$ & $4.4210629$& $2.6949834$ & $0.2533034$\\
\hline
$\exp(46.2)$ & $\exp(50)$ & $1.4190314$ & $4.3167121 $ & $2.6279248 $ & $0.2455362$ \\
\hline
$\exp(50)$ & $\exp(60)$ & $1.3227366 $ & $3.9675279 $ & $2.3928627$& $0.2159253$ \\
\hline
$\exp(60)$ & $\exp(70)$ & $1.1443596 $ & $3.3303140 $ & $1.9746460$& $0.1634613$ \\
\hline
$\exp(70)$ & $\exp(80)$ & $1.0275277$ & $2.9044834$ & $1.6974785$& $0.1318565$ \\
\hline
$\exp(80)$ & $\exp(90)$ & $0.9414561$ & $2.6078943 $ & $1.5092573 $& $0.1103318$\\
\hline
$\exp(90)$ & $\exp(100)$ & $0.8812460$ & $2.4151753$ & $1.3944136 $& $0.0986500$ \\
\hline
$\exp(100)$ & $\exp(110)$ & $0.8345856 $ & $2.2564675$ & $1.2987862 $& $0.0906041$ \\
\hline
$\exp(110)$ & $\exp(120)$ & $0.7956163 $ & $ 2.1317435$ & $1.2246505 $& $0.0835439$ \\
\hline
$\exp(120)$ & $\exp(130)$ & $0.7597190 $ & $2.0159946$ & $1.1545049 $& $0.0758754$\\
\hline
$\exp(130)$ & $\exp(140)$ & $0.7329142$ & $1.9273879$ & $1.1018686$ & $0.0711966$\\
\hline
$\exp(140)$ & $\exp(150)$ & $0.7086575$ & $1.8478283$ & $1.0532314$ & $0.0671401$\\
\hline
$\exp(150)$ & $\exp(160)$ & $0.6879899$ & $1.7855813$ & $1.0184406$ & $0.0633382$\\
\hline
$\exp(160)$ & $\exp(170.2)$ & $0.6693878$ & $1.7239944$ & $0.9794391$ & $0.0599159$\\
\hline
$\exp(170.2)$ & $\exp(500)$ & $0.6496901 $ & $1.6680151 $ & $0.9471296$ & $0.0558721$\\
\hline
$\exp(500)$ & $\exp(1000)$ & $0.4648597$ & $1.0757433 $ & $0.5873406 $ & $0.0212707 $\\
\hline
$\exp(1000)$ & $\exp(1500)$ & $0.4140847 $ & $0.9099961 $ & $0.4838796$ & $0.0123507$\\
\hline
$\exp(1500)$ & $\exp(2000)$ & $ 0.3945241$ & $0.8417122 $ & $0.4385802 $ & $0.0077839 $\\
\hline
$\exp(2000)$ & $\exp(2500)$ & $0.3814614$ & $0.8045939 $ & $0.4158624 $ & $0.0048886 $\\
\hline
$\exp(2500)$ & $\exp(3000)$ & $0.3756629$ & $0.7880794$ & $0.4059352$ & $0.0039808$\\
\hline
$\exp(3000)$ & $\exp(481958)$ & $0.3750548$ & $0.7752255$ & $0.3959471$ & $0.0049674 $\\
\hline
\end{tabular}
\label{uxvw}
\end{table}

\begin{table}[h]
\def\arraystretch{1.3}
\centering
\caption{Values of constants $d_1,d_2,d_3,d_4,d_5$ }
\begin{tabular}{|c|c|c|c|c|c|c|c|c|}
\hline
$T_0$ & $T_1$ & $d_{1,1}$ & $d_{1,2}$ & $d_{2,1}$ & $d_{2,2}$ & $d_3$ & $d_4$ & $d_5$  \\
\hline
$3\cdot 10^{12} $ & $\exp(29)$ & $0.107$ & $0.130$ & $0.027$ & $0.030$ & $0.003$ & $84.796$&$1.104$\\
\hline
$\exp(29)$ & $\exp(30)$ & $0.107$ & $0.130$& $0.027$& $0.030$& $0.003$& $84.570$& $1.104$ \\
\hline
$\exp(30)$ & $\exp(31)$ & $0.107$ & $0.130$& $0.027$& $0.030$& $0.003$& $86.156$& $1.104$ \\
\hline
$\exp(31)$ & $\exp(32)$ & $0.107$ & $0.130$& $0.027$& $0.030$& $0.003$& $87.265$& $1.106$ \\
\hline
$\exp(32)$ & $\exp(33)$ & $0.107$ & $0.131$& $0.027$& $0.030$& $0.003$& $88.852$& $1.107$ \\
\hline
$\exp(33)$ & $\exp(34)$ & $0.107$ & $0.131$& $0.027$& $0.030$& $0.003$& $90.014$& $1.109$ \\
\hline
$\exp(34)$ & $\exp(35)$ & $0.107$ & $0.132$& $0.027$& $0.030$& $0.003$& $90.957$& $1.110$ \\
\hline
$\exp(35)$ & $\exp(36)$ & $0.107$ & $0.132$& $0.027$& $0.030$& $0.002$& $92.011$& $1.111$ \\
\hline
$\exp(36)$ & $\exp(37)$ & $0.107$ & $0.133$& $0.027$& $0.030$& $0.002$& $93.509$& $1.113$ \\
\hline
$\exp(37)$ & $\exp(38)$ & $0.107$ & $0.133$& $0.027$& $0.030$& $0.003$& $95.141$& $1.115$ \\
\hline
$\exp(38)$ & $\exp(39)$ & $0.107$ & $0.134$& $0.027$& $0.031$& $0.003$& $96.691$& $1.117$ \\
\hline
$\exp(39)$ & $\exp(40)$ & $0.107$ & $0.134$& $0.027$& $0.031$& $0.002$& $97.416$& $1.118$ \\
\hline
$\exp(40)$ & $\exp(41)$ & $0.107$ & $0.134$& $0.027$& $0.031$& $0.002$& $98.881$& $1.119$ \\
\hline
$\exp(41)$ & $\exp(42)$ & $0.107$ & $0.135$& $0.027$& $0.031$& $0.002$& $99.901$& $1.120$ \\
\hline
$\exp(42)$ & $\exp(43)$ & $0.107$ & $0.135$& $0.027$& $0.031$& $0.002$& $100.416$& $1.121$ \\
\hline
$\exp(43)$ & $\exp(44)$ & $0.107$ & $0.135$& $0.028$& $0.031$& $0.002$& $102.078$& $1.123$ \\
\hline
$\exp(44)$ & $\exp(45)$ & $0.107$ & $0.136$& $0.028$& $0.031$& $0.002$& $104.242$& $1.124$ \\
\hline
$\exp(45)$ & $\exp(46.2)$ & $0.107$ & $0.136$& $0.028$& $0.031$& $0.002$& $104.616$& $1.125$ \\
\hline
$\exp(46.2)$ & $\exp(50)$ &  $0.107 $ & $0.136 $ & $0.028 $& $0.031$ & $0.002$ & $ 106.106$ & $1.126$\\
\hline
$\exp(50)$ & $\exp(60)$ & $0.108 $ & $0.138 $ & $0.028 $& $0.031$ & $0.002$ & $112.559$ & $1.132$\\
\hline
$\exp(60)$ & $\exp(70)$ & $0.108 $ & $0.142 $ & $0.028 $& $0.032$ & $0.002$ & $128.201$ & $1.146$\\
\hline
$\exp(70)$ & $\exp(80)$ & $0.108 $ & $0.145 $ & $0.028 $& $0.032$ & $0.002$ & $146.704$ & $1.156$\\
\hline
$\exp(80)$ & $\exp(90)$ & $0.108 $ & $0.147 $ & $0.028 $& $0.032$ & $0.002$ & $162.542$ & $1.164$\\
\hline
$\exp(90)$ & $\exp(100)$ & $0.107 $ & $0.147 $ & $0.027 $& $0.032$ & $0.002$ & $170.404$ & $1.163$\\
\hline
$\exp(100)$ & $\exp(110)$ & $0.107 $ & $0.146 $ & $0.027 $& $0.032$ & $0.002$ & $ 181.448$ & $1.159$\\
\hline
$\exp(110)$ & $\exp(120)$ & $0.107 $ & $0.145 $ & $0.027 $& $0.032$ & $0.002$ & $189.419$ & $1.157$\\
\hline
$\exp(120)$ & $\exp(130)$ & $0.107 $ & $0.146 $ & $0.027 $& $0.032$ & $0.002$ & $199.542$ & $1.158$\\
\hline
$\exp(130)$ & $\exp(140)$ & $0.107$ & $0.145$ & $0.027$ & $0.032$ & $0.002$ & $208.767$ & $1.156$\\
\hline
$\exp(140)$ & $\exp(150)$ & $0.107$ & $0.144$ & $0.027$ & $0.032$ & $0.002$ & $219.463$ & $1.153$\\
\hline
$\exp(150)$ & $\exp(160)$ & $0.106$ & $0.144$ & $0.027$ & $0.032$ & $0.002$ & $223.399$ & $1.151$\\
\hline
$\exp(160)$ & $\exp(170.2)$ & $0.106$ & $0.143$ & $0.027$ & $0.032$ & $0.002$ & $235.581$ & $1.150$\\
\hline
$\exp(170.2)$ & $\exp(500)$ & $0.106$ & $0.144$ & $0.027$ & $0.032$ & $0.002$ & $239.663$ & $1.151$\\
\hline
$\exp(500)$ & $\exp(1000)$ & $0.104$ & $0.139$ & $0.027$ & $0.031$ & $0.001$ & $575.833$ & $1.135$\\
\hline
$\exp(1000)$ & $\exp(1500)$ & $0.103$ & $0.133$ & $0.026$ & $0.030$ & $0.001$ & $ 1252.961$ & $1.116$\\
\hline
$\exp(1500)$ & $\exp(2000)$ & $0.103$ & $0.135$ & $0.026$ & $0.030$ & $0.001$ & $2659.132$ & $1.122$\\
\hline
$\exp(2000)$ & $\exp(2500)$ & $0.103$ & $0.142$ & $0.026$ & $0.031$ & $0.001$ & $3968.354$ & $1.147$\\
\hline
$\exp(2500)$ & $\exp(3000)$ & $0.103$ & $0.142$ & $0.026$ & $0.031$ & $0.001$ & $4910.047$ & $1.144$\\
\hline
$\exp(3000)$ & $\exp(481958)$ & $0.102$ & $0.128$ &  $0.026 $ & $0.029$ & $ 0.001$ & $9955.727$ & $1.097$\\
\hline
\end{tabular}
\label{dconstants}
\end{table}
\begin{table}[h]
\def\arraystretch{1.3}
\centering
\caption{Values of constants $c_1,c_2,c_3,c_4,c_5$ }
\begin{tabular}{|c|c|c|c|c|c|c|c|}
\hline
$T_0$ & $T_1$ & $c_1$ & $c_2$ & $c_3$ & $c_4$ & $c_5$  \\
\hline
$3\cdot 10^{12} $ & $\exp(29)$ &  $0.981$ & $195.230$ & $2222.717$ & $286.182$ & $0.062$ \\
\hline
$\exp(29)$ & $\exp(30)$ &  $0.981$& $195.734$ & $2243.325$ & $286.921$ & $0.061$\\
\hline
$\exp(30)$ & $\exp(31)$ &  $0.983$& $207.656$ & $2395.206$ & $304.415$ & $0.062$\\
\hline
$\exp(31)$ & $\exp(32)$ &  $0.984$& $220.851$ & $2563.120$ & $323.779$ & $0.064$\\
\hline
$\exp(32)$ & $\exp(33)$ &  $0.984$& $235.760$ & $2752.392$ & $345.657$ & $0.064$\\
\hline
$\exp(33)$ & $\exp(34)$ &  $0.984$& $250.519$ & $2941.458$ & $367.319$ & $0.066$\\
\hline
$\exp(34)$ & $\exp(35)$ &  $0.984$& $ 263.991$ & $3116.775$ & $387.084$ & $0.066$\\
\hline
$\exp(35)$ & $\exp(36)$ &  $0.984$& $278.761$ & $3308.747$ & $408.766$ & $0.066$\\
\hline
$\exp(36)$ & $\exp(37)$ &  $0.983$& $297.070$ & $3544.300$ & $435.641$ & $0.068$\\
\hline
$\exp(37)$ & $\exp(38)$ &  $0.983$& $317.207$ & $3803.499$ & $465.199$ & $0.069$\\
\hline
$\exp(38)$ & $\exp(39)$ &  $0.983$& $335.811$ & $4046.112$ & $492.513$ & $0.070$\\
\hline
$\exp(39)$ & $\exp(40)$ &  $0.982$& $350.391$ & $4241.658$ & $513.913$ & $0.071$\\
\hline
$\exp(40)$ & $\exp(41)$ &  $0.982$& $369.491$ & $4493.302$ & $541.959$ & $0.072$\\
\hline
$\exp(41)$ & $\exp(42)$ &  $0.982$& $385.420$ & $4707.817$ & $565.341$ & $0.072$\\
\hline
$\exp(42)$ & $\exp(43)$ &  $0.982$& $401.673$ & $4927.512$ & $589.199$ & $0.073$\\
\hline
$\exp(43)$ & $\exp(44)$ &  $0.981$& $424.491$ & $5229.292$ & $622.708$ & $0.074$\\
\hline
$\exp(44)$ & $\exp(45)$ &  $0.981$& $450.335$ & $5570.341$ & $660.661$ & $0.075$\\
\hline
$\exp(45)$ & $\exp(46.2)$ &  $0.981$& $466.658$ & $5799.757$ & $684.629$ & $0.075$\\
\hline
$\exp(46.2)$ & $\exp(50)$ & $0.981 $ & $488.558 $ & $6132.446 $ & $716.781$ & $0.074$\\
\hline
$\exp(50)$ & $\exp(60)$ & $0.979 $ & $592.937$ & $7615.822 $& $870.132$&$0.081$\\
\hline
$\exp(60)$ & $\exp(70)$ & $0.977 $ & $903.661 $ & $11836.516 $& $1326.769$ & $0.085$\\
\hline
$\exp(70)$ & $\exp(80)$ & $0.974 $ & $1293.789$ & $17238.260 $& $1900.434$ & $0.090$\\
\hline
$\exp(80)$ & $\exp(90)$ & $0.972 $ & $1725.506 $ & $23340.471 $& $2535.355$ & $0.096$\\
\hline
$\exp(90)$ & $\exp(100)$ & $0.973 $ & $2021.343 $ & $27715.348$& $2970.040$ & $0.093$\\
\hline
$\exp(100)$ & $\exp(110)$ & $0.972 $ & $2336.236 $ & $ 32428.941 $& $3432.726$ & $0.095$\\
\hline
$\exp(110)$ & $\exp(120)$ & $0.973 $ & $ 2639.408$ & $37050.947 $& $3878.189$& $0.093$\\
\hline
$\exp(120)$ & $\exp(130)$ & $0.972 $ & $3065.682 $ & $43481.947$& $4504.530$& $0.092$\\
\hline
$\exp(130)$ & $\exp(140)$ &  $0.972$& $3409.926$ & $48829.688$ & $5010.341$ & $0.089$\\
\hline
$\exp(140)$ & $\exp(150)$ &  $0.972$& $3792.749$ & $54797.783$ & $5572.839$ & $0.093$\\
\hline
$\exp(150)$ & $\exp(160)$ &  $0.971$& $4086.320$ & $59533.271$ & $6004.194$ & $0.088$\\
\hline
$\exp(160)$ & $\exp(170.2)$ &  $0.974$& $4549.556$ & $66812.743$ & $6684.845$ & $0.088$\\
\hline
$\exp(170.2)$ & $\exp(500)$ &  $0.972 $& $4968.978 $ & $82847.292 $ & $7301.119 $ & $0.093 $\\
\hline
$\exp(500)$ & $\exp(1000)$ & $0.965 $ & $30904.962 $ & $555939.909 $ & $45409.906$ & $ 0.086 $ \\
\hline
$\exp(1000)$ & $\exp(1500)$ & $0.973 $ & $113906.164$ & $2137924.019$ & $167366.915$ & $0.079 $\\
\hline
$\exp(1500)$ & $\exp(2000)$ & $0.988$ & $385838.839$ & $7457165.183$ & $566928.551$ & $0.040$ \\
\hline
$\exp(2000)$ & $\exp(2500)$ &  $0.965$& $936749.186$ & $18512203.901$ & $1376403.319$ & $0.094$\\
\hline
$\exp(2500)$ & $\exp(3000)$ &  $0.975$& $1419982.450$ & $28568596.811$ & $2086437.421$ & $0.100$\\
\hline
$\exp(3000)$ & $\exp(481958)$ & $0.984$ & $2211776.039$ &  $67229207.111$ & $3249851.641$ & $ 0.131$\\
\hline
\end{tabular}
\label{cconstants}
\end{table}  

\clearpage
\printbibliography

@article{thorner2023explicit,
      title={An explicit version of Bombieri's log-free density estimate and S\'ark\"ozy's theorem for shifted primes}, 
      author={Jesse Thorner and Asif Zaman},
journal={Forum Math.},
note={To appear },
      year={2024},
}

@article{HASANALIZADE2022219,
title = {Counting zeros of the Riemann zeta function},
journal = {J. Number Theory},
volume = {235},
pages = {219-241},
year = {2022},
author = {Elchin Hasanalizade and Quanli Shen and Peng-Jie Wong},
}

@article{TRUDGIAN2014280,
title = {An improved upper bound for the argument of the Riemann zeta-function on the critical line II},
journal = {J. Number Theory},
volume = {134},
pages = {280-292},
year = {2014},
author = {Timothy S. Trudgian},
}

@article{CullyHugill2019TwoED,
  title={Two explicit divisor sums},
  author={Michaela Cully-Hugill and Tim Trudgian},
  journal={Ramanujan J.},
  year={2019},
  volume={56},
  pages={141 - 149}
}

@book{olver_asymptotics_1974,
title = {Asymptotics and Special Functions},
	publisher = {Academic Press, New York},
	author = {Olver, Frank W. J.},
	year = {1974},
}

@book{Ivic2003TheRZ,
  title={The Riemann Zeta-Function : Theory and Applications},
  author={Aleksandar Ivi\'{c}},
  year={2003},
series={Dover Books on Mathematics Series},
publisher={	Dover Publications, Incorporated}
}

@article{Ramare15,
  author       = {Olivier Ramar{\'{e}}},
  title        = {Explicit estimates on several summatory functions involving the Moebius
                  function},
  journal      = {Math. Comp.},
  volume       = {84},
  number       = {293},
  pages        = {1359--1387},
  year         = {2015},
}

@article{ramare_explicit_2019,
	title = {Explicit average orders: news and problems},
	volume = {118},
	journal = {Banach Center Publ.},
	author = {Ramaré, Olivier},
	year = {2019},
	pages = {153--176},
}

@article{GRAHAM197883,
title = {An asymptotic estimate related to Selberg's sieve},
journal = {J. Number Theory},
volume = {10},
number = {1},
pages = {83-94},
year = {1978},
author = {S Graham},
}

@article{ramare_explicit_2017,
	title = {Explicit averages of non-negative multiplicative functions: going beyond the main term},
	volume = {147},
	number = {2},
	journal = {Colloq. Math.},
	author = {Ramaré, O. and Akhilesh, P.},
	year = {2017},
	pages = {275--313},
}

@article{ford_zero_2022,
    author = {Ford, K.},
    title = {Zero-free regions for the {R}iemann zeta function},
    note = {Preprint available at arXiv:1910.08205},
    year = {2022}
}

@article {platt_riemann_2021,
    AUTHOR = {Platt, D. J. and Trudgian, T. S.},
     TITLE = {The {R}iemann hypothesis is true up to {$3\cdot 10^{12}$}},
   JOURNAL = {Bull. Lond. Math. Soc.},
  FJOURNAL = {Bulletin of the London Mathematical Society},
    VOLUME = {53},
      YEAR = {2021},
    NUMBER = {3},
     PAGES = {792--797},
}

@article{bellotti2023explicit,
      title={Explicit zero density estimate near unity}, 
      author={Chiara Bellotti},
      year=2023,
month=Nov,
     note={arXiv:2311.05136},
}

@article{ramare2016explicit,
  title={An explicit density estimate for Dirichlet L-series},
  author={Ramar{\'e}, Olivier},
  journal={Math. Comp.},
  volume={85},
  number={297},
  pages={325--356},
  year={2016}
}

@article{KADIRI201822,
title = {Explicit zero density for the Riemann zeta function},
journal = {J. Math. Anal. Appl.},
volume = {465},
number = {1},
pages = {22-46},
year = {2018},
author = {Habiba Kadiri and Allysa Lumley and Nathan Ng},

}

@article{ingham1937difference,
  title={On the difference between consecutive primes},
  author={Ingham, Albert Edward},
  journal={Q. J. Math.},
  number={1},
  pages={255--266},
  year={1937},
  publisher={Oxford University Press}
}

@article{trudgian2023optimal,
      title={On optimal exponent pairs}, 
      author={Timothy S. Trudgian and Andrew Yang},
journal={Math. Comp.},
     note={To appear},
year={2024},
}

@article{pintz2023remark,
      title={A remark on density theorems for Riemann's zeta-function}, 
      author={Janos Pintz},
      year={2023},
      month=sep,
      note={arXiv:2310.04544}
}

@article{Ingham1940ONTE,
  title={On the estimation of $N(\sigma, T)$},
  author={Albert Edward Ingham},
  journal={Q. J. Math.},
  year={1940},
  volume={1},
  pages={201-202},
}

@article{Huxley1971/72,
author = {Huxley, M.N.},
journal = {Invent. Math.},
pages = {164-170},
title = {On the Difference between Consecutive Primes.},
volume = {15},
year = {1971/72},
}

@article{patel2023explicit,
      title={An explicit sub-Weyl bound for $\zeta(1/2 + it)$}, 
      author={Dhir Patel and Andrew Yang},
      year={2024},
journal={J. Number Theory},
note={To appear },
}

@article{hiary_improved_2024,
	title = {An improved explicit estimate for $\zeta(1/2 + it)$},
	volume = {256},
	journal = {J. Number Theory},
	author = {Hiary, Ghaith A. and Patel, Dhir and Yang, Andrew},
	year = {2024},
	pages = {195--217},
}

@article{hiary_explicit_2016,
	title = {An explicit van der {Corput} estimate for  $\zeta( 1 / 2 + i t )$},
	volume = {27},
	number = {2},
	journal = {Indag. Math.},
	author = {Hiary, Ghaith A.},
	year = {2016},
	pages = {524--533},
}

@article{jutila77,
 author = {Matti Jutila},
 journal = {Math. Scand.},
 number = {1},
 pages = {45--62},
 title = {On Linnik's Constant},
 volume = {41},
 year = {1977}
}

@article{BELLOTTI2024128249,
title = {Explicit bounds for the Riemann zeta function and a new zero-free region},
journal = {J. Math. Anal. Appl.},
volume = {536},
number = {2},
pages = {128249},
year = {2024},
author = {Chiara Bellotti},
}

@article{YANG2024128124,
title = {Explicit bounds on $\zeta(s)$ in the critical strip and a zero-free region},
journal = {J. Math. Anal. Appl.},
volume = {534},
number = {2},
pages = {128124},
year = {2024},
author = {Andrew Yang},
}

@article{mossinghoff2024explicit,
      title={Explicit zero-free regions for the Riemann zeta-function}, 
      author={Michael J. Mossinghoff and Timothy S. Trudgian and Andrew Yang},
      journal={Res. Number Theory},
      year={2024},
       volume={10},
number={11},
}

@article{pintz_new_2019,
	title = {Some new density theorems for {Dirichlet} $L$-functions},
	volume = {118},
	journal = {Banach Center Publ.},
	author = {Pintz, János},
	year = {2019},
	pages = {231--244},
}

@article{Bourgain2000OnLV,
  title={On large values estimates for Dirichlet polynomials and the density hypothesis for the Riemann zeta function},
  author={Jean Bourgain},
  journal={Int. Math. Res. Not.},
  year={2000},
  volume={2000},
  pages={133-146},
}

@article{SIMONIC2020124303,
title = {Explicit zero density estimate for the Riemann zeta-function near the critical line},
journal = {J. Math. Anal. Appl.},
volume = {491},
number = {1},
pages = {124303},
year = {2020},
author = {Aleksander Simonič},
}

@article{Lin44,
title = {On the least prime in an arithmetic progression. I. The basic theorem},
journal = {Rec. Math. [Mat. Sbornik] N.S.},
volume = {15},
number = {57},
issue ={2},
pages = {139-178},
year = {1944},
author = { Yurii Vladimirovich Linnik},
}

@article{Lin44a,
title = {On the least prime in an arithmetic progression. II: The Deuring- Heilbronn theorem},
journal = {Rec. Math. [Mat. Sbornik] N.S.},
volume = {57},
number = {3},
pages = {347–368},
year = {1944},
author = {Yurii Vladimirovich Linnik },
}

@article{Fogels1965,
author = {Fogels, E.},
journal = {Acta Arith.},
number = {1},
pages = {67-96},
title = {On the zeros of $L$-Functions},
volume = {11},
year = {1965},
}

@article{Turan61,
author = {Tur{\'a}n, P.},
journal = {Magyar Tud. Akad. Mat. Kutat{\'o} Int. K{\"a}zl.,},
pages = {165-170},
title = {On a density theorem of Yu. V. Linnik},
volume = {6},
year = {1961},
}

@book{bombiericrible,
     author = {Bombieri, Enrico},
     title = {Le grand crible dans la th\'eorie analytique des nombres},
     series = {Ast\'erisque},
     publisher = {Soci\'et\'e math\'ematique de France},
     number = {18},
     year = {1974},
}

@article{gallagher_large_1970,
	title = {A large sieve density estimate near $\sigma=1$},
	volume = {11},
	number = {4},
	journal = {Invent. Math.},
	author = {Gallagher, P. X.},
	year = {1970},
	pages = {329--339},
}

@article{graham_linniks_1981,
	title = {On {Linnik}'s constant},
	volume = {39},
	number = {2},
	journal = {Acta Arith.},
	author = {Graham, S.},
	year = {1981},
	pages = {163--179},
}

@article{hb92,
author = {Heath-Brown, D. R.},
title = {Zero-Free Regions for Dirichlet $L$-Functions, and the Least Prime in an Arithmetic Progression},
journal = {Proc. Lond. Math. Soc. },
volume = {s3-64},
number = {2},
pages = {265-338},
year = {1992}
}

@article{yangdj,
title = {Some explicit estimates for the error term in the prime number theorem},
journal = {J. Math. Anal. Appl.},
volume = {527},
number = {2},
pages = {127460},
year = {2023},
author = {Daniel R. Johnston and Andrew Yang},
}

@article{fiori_sharper_2023,
	title = {Sharper bounds for the error term in the prime number theorem},
	volume = {9},
	number = {3},
	journal = {Res. Number Theory},
	author = {Fiori, Andrew and Kadiri, Habiba and Swidinsky, Joshua},
	year = {2023},
	pages = {63},
}

@article{michaela2024error,
      title={On the error term in the explicit formula of Riemann-von Mangoldt II}, 
      author={Daniel R. Johnston and Michaela Cully-Hugill},
      year={2024},
      note={arXiv:2402.04272},
}

@article{CULLYHUGILL2023100,
title = {Primes between consecutive powers},
journal = {J. Number Theory},
volume = {247},
pages = {100-117},
year = {2023},
author = {Michaela Cully-Hugill},
}

@article{michaelavonmangoldt,
author = {Cully-Hugill, Michaela and Johnston, Daniel R.},
title = {On the error term in the explicit formula of Riemann–von Mangoldt},
journal = {Int. J. Number Theory},
volume = {19},
number = {06},
pages = {1205-1228},
year = {2023},
}

@article{platt_error_2020,
	title = {The error term in the prime number theorem},
	volume = {90},
	number = {328},
	journal = {Math. Comp.},
	author = {Platt, David J. and Trudgian, Timothy S.},
	year = {2020},
	pages = {871--881},
}

@article{selberg1946contribution,
  title={Contribution to the theory of the Riemann zeta-function},
  author={Selberg, Atle},
  journal={Arch. Math. Naturvidensk.},
  year={1946},
volume={48},
number={5},
pages={89-155},
}

@online{Bellotti_Log-free ,
author = {Bellotti, Chiara},
title = {Log-free zero density},
note={GitHub repository available at},
url={https://github.com/ChiaraBellotti/Log-free-zero-density-estimates-for-zeta},
}

@article{ramaré2024l2bound,
      title={An $L^2$-bound for the Barban-Vehov weights}, 
      author={Ramaré, Olivier and Zuniga Alterman, Sebastian },
      year={2024},
      note={arXiv:2405.12662},
}
\end{document}